\documentclass[oneside]{amsart}

\usepackage[T1]{fontenc}

\usepackage{amsmath,amsthm,amsfonts,amssymb,latexsym, mathrsfs}

\makeatletter

\theoremstyle{plain}
\newtheorem{thm}{Theorem}[section]

\numberwithin{equation}{section} 

\numberwithin{figure}{section} 

\theoremstyle{plain}
\newtheorem{cor}[thm]{Corollary} 

\theoremstyle{definition}
\newtheorem{defn}[thm]{Definition}

\theoremstyle{plain}
\newtheorem{lem}[thm]{Lemma} 

\theoremstyle{plain}
\newtheorem{prop}[thm]{Proposition} 

\theoremstyle{plain}
\newtheorem{fact}[thm]{Fact}

\theoremstyle{plain}

\theoremstyle{plain}

\newtheorem{rem}[thm]{Remark}

\newtheorem{example}[thm]{Example}

\newtheorem{ques}[thm]{Question}

\newtheorem{obse}[thm]{Observation}

\newtheorem{theo}[thm]{Theorem}

\newtheorem{rema}[thm]{Remark}

\newtheorem{nota}[thm]{Notation}

\ifx\pdfoutput\@undefined\usepackage[usenames,dvips]{color}

\else\usepackage[usenames,dvipsnames]{color}

\IfFileExists{pdfcolmk.sty}{\usepackage{pdfcolmk}}{}

\fi

\makeatletter

\makeatletter
\def\moverlay{\mathpalette\mov@rlay}
\def\mov@rlay#1#2{\leavevmode\vtop{%
   \baselineskip\z@skip \lineskiplimit-\maxdimen
   \ialign{\hfil$\m@th#1##$\hfil\cr#2\crcr}}}
\newcommand{\charfusion}[3][\mathord]{
    #1{\ifx#1\mathop\vphantom{#2}\fi
        \mathpalette\mov@rlay{#2\cr#3}
      }
    \ifx#1\mathop\expandafter\displaylimits\fi}
\makeatother

\newcommand{\cupdot}{\charfusion[\mathbin]{\cup}{\cdot}}

\usepackage{amssymb}

\usepackage{amsmath}

\def \<{\langle}
\def \>{\rangle}

\usepackage{amsfonts}

\usepackage{amscd}

\usepackage{latexsym}

\usepackage{epsfig}

\makeatother

\def\Ind#1#2{#1\setbox0=\hbox{$#1x$}\kern\wd0\hbox to 0pt{\hss$#1\mid$\hss}
\lower.9\ht0\hbox to 0pt{\hss$#1\smile$\hss}\kern\wd0}
\def\ind{\mathop{\mathpalette\Ind{}}}
\def\Notind#1#2{#1\setbox0=\hbox{$#1x$}\kern\wd0\hbox to 0pt{\mathchardef
\nn=12854\hss$#1\nn$\kern1.4\wd0\hss}\hbox to
0pt{\hss$#1\mid$\hss}\lower.9\ht0 \hbox to
0pt{\hss$#1\smile$\hss}\kern\wd0}
\def\nind{\mathop{\mathpalette\Notind{}}}

\newcommand{\bcl}{\operatorname{bcl}}

\newcommand{\acl}{\operatorname{acl}}
\newcommand{\ecl}{\operatorname{ecl}}
\newcommand{\spa}{\operatorname{span}}

\newcommand{\tp}{\operatorname{tp}}

\newcommand{\qftp}{\operatorname{qftp}}
\newcommand{\Lstp}{\operatorname{Lstp}}

\newcommand{\cl}{\operatorname{cl}}

\newcommand{\scl}{\operatorname{scl}}

\begin{document}

\title{Dense-codense expansions of quasiminimal pregeometry structures}

 \date{\today}

\author{Alexander Berenstein}
\address{Universidad de los Andes,
Cra 1 No 18A-12, Edificio H, Bogot\'{a}, Colombia}
 \urladdr{www.matematicas.uniandes.edu.co/\textasciitilde aberenst}

\author{Evgueni Vassiliev}
\address{Grenfell Campus, Memorial University of Newfoundland, Corner Brook, NL A2H 6P9, Canada}
\email{yvasilyev@grenfell.mun.ca}

\thanks{The authors would like to thank Tapani Hyttinen, Jonathan Kirby and Andr\'es Villaveces for valuable feedback. The first author would like to thank Memorial University and NSERC for support during an academic visit in Spring 2025 and Universidad de los Andes for supporting a STAI in Spring 2025.}

\begin{abstract}
We study expansions of quasiminimal pregeometry structures with a dense codense unary predicate and their relation with the complexity properties of the pregeometry of the underlying  structure. We consider beautiful pairs as well as $H$-structures. We show each of these expansions can be axiomatized with a single $\mathcal{L}_{\omega_1 \omega}(Q)$-sentence and that both expansions are $\omega$-stable. For $H$-structures we provide a natural notion of independence in the expansion and when the underlying structure is modular, we also provide a natural notion of independence for beautiful pairs. Then we relate the complexity of the pregeometry to properties of the expansions.

\end{abstract}

\pagestyle{plain}

\maketitle

\section{Introduction}

Quasiminimal classes of structures (see for example \cite{BHHKK, Ki}) forms a good generalization of strongly minimal theories, they have a closure operator defining a pregeometry and they have a unique generic type, but instead of being an elementary class, they can be described by an  $\mathcal{L}_{\omega_1 \omega}(Q)$-sentence (see for example Theorem 2.3 in \cite{BHHKK}). 
Another natural setting that generalizes strongly minimal theories are the geometric theories, those complete first order theories where the algebraic closure has the exchange property and that eliminate the quantifier $\exists^{\infty}$. Besides strongly minimal theories, this family includes $SU$-rank one theories as well as dense o-minimal theories and it is a reasonable setting for studying the complexity of pregeometries (see for example \cite{BeVaLP,BeVaw}). 

Generalizing ideas from the strongly minimal setting to quasiminimal classes has been very fruitful. For example, in \cite{HK} the authors show how to generalize the group configuration when the underlying quasiminimal pregeometry structure has a non-trivial pregeometry, they study plane curves and even recover some form of canonical bases; a study of local modularity is carried out in \cite{H05} and some of the consequences are presented in \cite{HK}. In this paper we construct an example (see Example \ref{vector-space}) of an non-trivial locally modular quasiminimal pregeometry structure that cannot be interpreted as a strongly minimal theory, showing that the setting includes interesting new families of locally modular structures. Namely, we consider an infinite-dimensional projective geometry 
over an infinite division ring $\mathbb{F}$, i.e., the associated geometry $(PG(V),\cl^*)$ of an $\mathbb{F}$-vector space $V$, with the $n$-ary relations $R_{\vec \lambda}$ indicating that for some representatives $v_1,\ldots, v_n\in V$ of the given points of the geometry the linear combination with coefficients given by $\vec \lambda=(\lambda_1,\ldots,\lambda_n)$ we have $\lambda_1v_1+\ldots+\lambda_nv_n=0$.

Similarly, a lot of attention has been paid to generalizing ideas from the strongly minimal setting to the geometric one, in particular how to characterize linearity \cite{BeVaw, v2003}. A key tool in this process has been the use of expansions, mostly beautiful pairs \cite{BeVaLP} and $H$-structures \cite{BeVa}.

The study of elementary pairs (and similar unary predicate expansions) goes back to the classical work of Robinson \cite{Ro} on the theory of real closed fields with a predicate for a dense real closed subfield.  
In \cite{Ma}, Macintyre had isolated a key property that generalized the density condition in such pairs: given a pair $(M,P)$, any infinite definable subset of  $M$  has a nonempty intersection with $P(M)$. Poizat \cite{Po} introduced a well-behaved class of elementary pairs of stable structures, the beautiful pairs, that produces a complete stable theory $T_P$ provided the base theory has non-finite cover property (nfcp); in this case, $T_P$ also has nfcp. 

Further connections of the properties of elementary pairs of a stable theory $T$ and stability-theoretic complexity of $T$ were explored by Bouscaren \cite{Bous}.  Buechler \cite{Bu} used beautiful pairs of strongly minimal structures to relate the complexity of the pregeometry associated to $\acl$ in models of $T$ to the U-rank of $T_P$. Generalizing the work of Robinson \cite{Ro}, van den Dries \cite{LvdD} studied dense pairs of o-minimal
theories that expand the theory of ordered abelian groups, showed that the theory of dense pairs is complete and
gave a description of definable sets.

In \cite{v2003}, the second named author introduced the notion of a generic pair of SU-rank 1 structures, resulting in a complete supersimple theory of pairs $T_P$, and used Buechler's approach to characterize linearity of $T$ in terms of the SU-rank of $T_P$ and various model-theoretic properties of these elementary pairs \cite{v2003, v2010}. The notions of beautiful pairs of stable structures and generic pairs of SU-rank 1 structures were generalized to the class of simple theories by Ben-Yaacov, Pillay and the second named author \cite{BPV}, resulting in the notion of lovely pairs. In \cite{BeVaLP}, we introduced a common generalization (which we called lovely, or dense-codense, pairs) of generic pairs of SU-rank 1 structures and dense pairs of o-minimal structures to the class of geometric structures. Such pairs satisfy the density property in the sense of Macintyre and the codensity (extension) property which requires that any infinite $A$-definable subset of $M$ has a realization in $M\backslash \acl(A\cup P(M))$. 

Lovely pairs share some properties with other similar "tame" unary expansions such as the algebraically and real closed fields expanded with a multiplicative subgroup having Mann property explored by van den Dries and Gunaydin \cite{vdDG}. Thorn-forking in such pairs in the real closed case was studied by the first author, Ealy and Gunaydin \cite{BEG}. 
In \cite{BeVa}, we introduced the notion of an H-structure, an expansion of a geometric structure with a dense-codense $\acl$-independent subset, generalizing the notion of an expansion of an o-minimal structure with a dense independent subset introduced by  Dolich, Miller and Steinhorn \cite{DMS}. Some well-behaved dense-codense
expansions of geometric structures "intermediate" between lovely pairs and H-structures were explored in \cite{BdEV, BeVaMult}. 

Pair constructions had many applications. In \cite{BPV}, lovely pairs were used to introduce and characterize the weak non-finite cover property, the appropriate generalization of the finite cover property in the setting of simple theories.
Lovely/generic pairs of geometric structures were used to study and characterize the appropriate notion of linear SU-rank 1/geometric structure in \cite{v2003,BeVaw}. Dense-codense expansions also serve as a source of examples in neostability theory \cite{BeKi, DK}.

In this paper, we will first find reasonable analogs of beautiful pairs and H-structures, when the underlying model is a quasiminimal pregeometry structure instead of a geometric structure. 
We show that both expansions in the setting of quasiminimal pregeometry structures is tame: they can be axiomatized with a single $\mathcal{L}_{\omega_1 \omega}(Q)$-sentence, both expansions are $\omega$-stable and we give a characterization of types in the new language. We also prove a natural generalization of superstability for both expansions, namely, types in the extended language do not split over a finite set. For $H$-structures, we provide a natural notion of independence in the expansion and show that it satisfies the usual properties for an independence relation, including stationarity. Similarly, we introduce a natural notion of independence in beautiful pairs of modular quasiminimal classes and prove the usual properties associated to an independence relation, including stationarity.  We define a notion of rank based on Lascar splitting, which we denote by $U_{Lsp}$. This will allow us to prove some results tying the complexity of the geometry with the ranks in the pairs, namely, we recover information of the complexity of pregeometry of the underlying class (desintegrated, locally modular, non-locally modular) generalizing ideas from \cite{Bu, v2003, v2010} using pairs. The main new contribution in this last part of the paper is that we found new proofs for these results that \emph{do not rely} on the use of canonical bases. In particular, in section \ref{sec:ranksgeometry} we prove:\\

\noindent \textbf{Theorem \ref{ranksHstruct}}
 Let $(M,H)$ be an $H$-structure associated to a quasiminimal class $\mathcal{C}$. Then $(M,\cl)$ is trivial iff $U_{Lsp}((M,H))= 1$ and $(M,\cl)$ is not trivial iff $U_{Lsp}((M,H))\geq \omega$. \\

\noindent \textbf{Theorem \ref{ranksbeautifulpairs}} Let $(M,P)$ be a beautiful pair associated to a quasiminimal class $\mathcal{C}$. Assume $(M,\cl)$ is trivial, then $U_{Lsp}((M,P))= 1$. Assume $(M,\cl)$ is locally modular non-trivial, then $U_{Lsp}((M,P))= 2$.\\

Finally, generalizing the results (see \cite{v2003}) of the second named author of this paper, we characterize local modularity in terms of properties of the corresponding beautiful pairs. As in previous work of the two authors on geometric structures, the condition of being weakly $1$-based (see Definition \ref{defn:w1b}) plays the role of $1$-basedness from strongly minimal theories.\\

\noindent \textbf{Theorem \ref{linearity}}
The following are equivalent for  a quasiminimal class $\mathcal{C}$:
\begin{enumerate}
    \item[(i)] $\mathcal{C}$ is weakly one-based.
    \item[(ii)] $\mathcal{C}$ is locally modular.
    \item[(iii)] In any sufficiently large beautiful pair $(M,P)$ associated to $\mathcal{C}$ we have $\bcl_P=\cl$.
    \item[(iv)] In any sufficiently large beautiful pair $(M,P)$ associated to $\mathcal{C}$ the localization $\cl(-\cup P(M))$ is modular.
    \item[(v)] Whenever $M,N\in \mathcal{C}$ and $(M,P)\subseteq (N,P)$ are beautiful pairs we have $(M,P)\preceq (N,P)$.
\end{enumerate}

This paper is organized as follows. In section \ref{quasiminimal:basics} we recall some facts from \cite{BHHKK, Ki} about quasiminimal pregeometry structures and their associated quasiminimal classes. In section \ref{sec:linearity} we introduce local modularity and weak 1- basedness (see definition) for quasiminimal pregeometry structures. In section \ref{sec:dense-codense}
we study $H$-structures and beautiful pairs of quasiminimal pregeometry structures and prove their main properties, such as axiomatizability in $\mathcal{L}_{\omega_1 \omega}(Q)$ and superstabilty. Section \ref{sec:ranks} deals with studying natural notions of independence in these expansions and their connections with non-Lascar splitting. Finally in section \ref{sec:ranksgeometry} we study ranks associated in these expansions and prove Theorem \ref{linearity}. 


\section{Preliminaries:  quasiminimal classes}\label{quasiminimal:basics}

In this section we review quasiminimal classes, we will use as background for quasiminimal classes the papers \cite{BHHKK, Ki} and assume the reader is familiar with the main results in the subject, for example we will use the fact that a quasimininal class is excellent \cite{BHHKK}. The main example we have in mind is complex exponentiation, one of the representative cases behind Zilber's work on the subject (see \cite{Zi}), but we will use repeatedly more accessible examples from \cite{Ki} to clarify our notions. For connections between quasiminimality and first order stability theory see \cite{Pillay-Tanovic}.

Let $M$ be a structure in a countable language $\mathcal{L}$ equipped with a pregeometry $\cl$. We write $\tp$ for the quantifier free type of a tuple. We assume that $M$
is a \emph{quasiminimal pregeometry structure} in the sense of \cite{BHHKK}. 
These axioms imply that we identify, for $\vec a\in M^n$, $\tp(\vec a)$ with 
$\tp_{\omega_1 \omega}(\vec a)$ and we will use any of the two expressions whenever we need them.
We will list below the axioms for quasiminimal pregeometry structures from \cite{BHHKK} as they will be crucial in this paper and we will need them for future reference. 

QM1. The pregeometry is determined by the language. That is, if $\tp(a,\vec b) =
\tp(a',\vec b')$ then $a \in \cl(\vec b)$ if and only if $a' \in \cl(\vec b')$.

QM2. $M$ is infinite-dimensional with respect to $\cl$.

QM3. (Countable closure property) If $A \subseteq M$ is finite then $\cl(A)$ is countable.

QM4. (Uniqueness of the generic type) Suppose that $A,A'\subseteq M$ are countable
closed subsets, enumerated such
that $\tp(A) =\tp(A')$. If $a \in M \setminus A$ and $a' \in M \setminus A'$ then $\tp(A, a) = \tp(A', a')$ (with respect to the same
enumerations for $A$ and $A'$).

QM5. ($\aleph_0$-homogeneity over closed sets and the empty set)
Let $A,A'\subseteq M$ be countable closed subsets or empty, enumerated such
that $\tp(A) =\tp(A')$, and let $\vec b$, $\vec b '$ be finite tuples from $M$ such that
$\tp(A,\vec b) = tp(A',\vec b   ')$, and let $a \in \cl(A,\vec{b})$. Then there is $a' \in M$ such
that $\tp(A,\vec{b}, a) = \tp(A',\vec b', a')$.

To emphasize that $M$ is a quasiminimal pregeometry structure we may write $(M,\cl)$ or even $(M,\cl_{M})$ instead of just $M$. We also need the following notions from \cite{BHHKK}. Given $(M_1,\cl_{M_1})$ and $(M_2,\cl_{M_2})$ both weakly quasiminimal pregeometry $\mathcal{L}$-structures, we
say that an $\mathcal{L}$-embedding $\theta :M_1 \to M_2$ is \emph{closed embedding} if for each $A \subseteq M_1$
we have $\theta(\cl_{M_1} (A)) = \cl_{M_2}(\theta(A))$. If this is the case, $\theta(M_1)$ is closed in $M_2$ with
respect to $\cl_{M_2}$. We write $M_1 \preceq M_2$ when the inclusion is a closed embedding.

Given a quasiminimal pregeometry structure $M$, we let $\mathcal{C}$ be the smallest class of $\mathcal{L}$-structures which contain $M$ and it is closed under the following operations: isomorphisms,  closed substructures, and under taking unions of chains of closed embeddings. We call any such class $\mathcal{C}$ a \emph{quasiminimal class}. We say $M$ is a \emph{weakly quasiminimal pregeometry structure} if it satisfies all the axioms except possibly $QM2$. So, for example, if we take the class $\mathcal{C}$ of all algebraically closed fields of characteristic zero in the language of rings, those fields with infinite trascendence degree are quasiminimal pregeometry structures, while those of finite trascendence degree are weakly quasiminimal pregeometry structures.

We will need the main result from \cite{BHHKK}:

\begin{fact}\label{Defclass} (Theorem 2.3 \cite{BHHKK}) If $\mathcal{C}$ is a quasiminimal class then every structure $A\in \mathcal{C}$ is a weakly quasiminimal pregeometry structure, and up to isomorphism there is
exactly one structure in $\mathcal{C}$ of each cardinal dimension. Furthermore, $\mathcal{C}$ is the class of models of an $\mathcal{L}_{\omega_1 \omega}(Q)$-sentence.
\end{fact}

\begin{nota}
Given a quasiminimal class    $\mathcal{C}$, sometimes we will write $\psi_{\mathcal{C}}$ for the sentence given in Fact \ref{Defclass}. 
\end{nota}

We will also need the following result that gives a definable control to the pregeometry:

\begin{fact}\label{definability-density} (Lemma 5.3 in \cite{Ki} $\mathcal{L}_{\omega_1 \omega}$-definability of the pregeometry). Let $\mathcal{C}$ be a quasiminimal excellent class. For each $n \in \mathbb{N}$ there is a quantifier-free $\mathcal{L}_{\omega_1,\omega}$-formula
$\pi_n(x, y_1,\dots,y_n)$ such that for each $M \in \mathcal{C}$ and each $a, b_1,\dots, b_n \in M$, we
have $a \in \cl_M(b_1,\dots,b_n)$ iff $M \models \pi_n(a,b_1,\dots,b_n)$    
\end{fact}

As stated earlier, given the results from \cite{BHHKK}, the hypothesis of excellence above is redundant. 

In the next section we will draw some analogies with strongly minimal theories. To do so, we need a good notion of independence.

\begin{nota}
 Let $\mathcal{C}$ be a quasiminimal class, let $M\in \mathcal{C}$ and let $\vec a\in M^n$.
 By $\dim(\vec a)$ we mean the dimension of the tuple $\vec a$ with respect to the pregeometry $(M,\cl)$. For $B\subset M$, we write $\dim(\vec a/B)$ for the dimension with respect to the pregeometry $(M,\cl_B)$, where $\cl_B$ is the localization of the pregeometry on the set $B$.
 For $B,C\subset M$ we write $\vec a\ind_C B$ if $\dim(\vec a/B\cup C)=\dim(\vec a/C)$, and $A\ind_CB$ if $\vec a\ind_C B$ for every finite tuple $\vec a$ in $A$. Sometimes, to empahsize that this notion of independence comes from the underlying pregeometry, we may refer to this notion of independence as $\cl$-independence and may write $\ind^{\cl}$ instead of $\ind$.
 \end{nota}

The notion of independence described above satisfies the usual properties of algebraic independence in a strongly minimal theory. One of the key properties is stationarity of types of tuples over models. It is stated in \cite{AHKK} and proved in \cite{Ka} for Lascar non-splitting in a framework that generalizes that of quasiminimal classes. The reader may want to check \cite[section 4]{Ka} where the author explains that the independence relation associated to Lascar non-splitting coincides with $\cl$-independence in $\mathcal{C}$. We state the fact for reference.


\begin{fact}\label{stationarityinC}
   The notion of independence $\ind$ satisfies stationarity over models in $\mathcal{C}$. 
\end{fact}

We will use a refined version of the previous result for Lascar strong types. We will assume the reader is familiar with Lascar strong types in the quasiminimal setting, the reader may want to check \cite{Ka}.

\begin{fact}[Stationarity, Lemma 53 \cite{Ka}] If $A \subseteq B\subseteq M$, $\vec a,\vec b\in M$, $\vec a \ind_A B$, $\vec b \ind_A B$ and $Lstp(\vec a/A) =
Lstp(\vec b/A)$, then
$Lstp(\vec a/B)=Lstp(\vec b/B)$.
\end{fact}

\section{Linearity in quasiminimal geometry structures}\label{sec:linearity}

In this section we introduce a notion of weak $1$-basedness and we prove that it is equivalent to local modularity. Locally modular quasiminimal pregeometry structures were studied previously in \cite{H05,HK}. We start with the definitions:

\begin{defn} 
Let $(M,\cl)$ be a quasiminimal pregeometry structure. We say that $(M,\cl)$ is modular if for all tuples $\vec a,\vec b\in M$ we have $\vec a\ind^{\cl}_{\cl(\vec a)\cap \cl(\vec b)} \vec b$. 

Let $\mathcal{C}$ be a quasiminimal class. We say $\mathcal{C}$ is \emph{modular} if for all $M\in \mathcal{C}$, the structure $(M,\cl)$ is modular.
\end{defn}

\begin{obse}
  Since a quasiminimal class $\mathcal{C}$ is categorical in every cardinal and the definition of modularity only dependns on finite tuples, $\mathcal{C}$ is modular if and only if for some
  $M\in \mathcal{C}$, the structure $(M,\cl)$ is modular.
\end{obse}

We recall the notion of projectivity and some equivalences. 

\begin{defn}(see Defn. 4.2.2 \cite{Bu})
    Let $(X,\cl)$ be a pregeometry. We say  projective if $(X,\cl)$ is nontrivial and for all $a, b \in X$ and $C\subseteq  X$ such that $a\in \cl(C\cup \{b\})$, there is a $c \in \cl(C)$ such that $a\in \cl(\{c, b\})$.
\end{defn}

\begin{fact} (\cite[Lemma 4.2.1]{Bu})
    A pregeometry is projective if and only if it is nontrivial and modular.
\end{fact}

Let $e\not \in \cl(\emptyset)$ and consider $(M,cl_e)$, the new quasiminimal pregeometry structure where we interpret $\cl_e(C)=\cl(C\cup \{e\})$. We call it the \emph{localization of $(M,cl)$ in $\{e\}$}.

\begin{defn} Let $(M,\cl)$ be a quasiminimal pregeometry structure. We say that $(M,\cl)$ is locally modular if for any $e\not \in \cl(\emptyset)$, the quasiminimal pregeometry structure $(M,cl_e)$ is modular.
\end{defn}

\begin{defn}\label{defn:w1b} Let $(M,\cl)$ be a quasiminimal pregeometry structure.
We say $(M,\cl)$ is \emph{weakly 1-based} if for any $A\subset M$ and $\vec c\in M^n$ and for any (some) $\vec d\in M^n$ satisfying $\Lstp(\vec d/A)=\Lstp(\vec c/A)$ and $\vec d\ind_A \vec c$  we have $\vec c\ind_{\vec d} A$.
\end{defn}

\begin{obse}
Note that in the previous definition, under the same assumptions, $\vec d\ind_A \vec c$ also implies $\vec d\ind_{\vec c} A$.
\end{obse}

\begin{theo}\label{w1bimplieslocmod}
Assume $(M,\cl)$ is \emph{weakly 1-based} and let
$e\in M$ be generic (so $e\not \in cl(\emptyset)$). Then the localization $(M,cl_e)$ is a quasiminimal modular structure.   
\end{theo}

\begin{proof}
    
We modify the argument in pages 144-145 from Buechler's book \cite{BueBook}, the original argument goes back to work of Zil'ber. The argument that appears in \cite{BueBook} uses canonical bases of plane curves, we find a way around them using instead that $(M,\cl)$ is weakly 1-based. Let $e\not \in \cl(\emptyset)$ and consider $(M,cl_e)$, the new quasiminimal pregeometry structure where we interpret $\cl_e(C)=\cl(C\cup \{e\})$. It is enough to show projectivity. Let $B\subset M$ be finite and let $a,b\in M$ with 
$a\in \cl(b,e,B)$. Our goal is to find $d\in \cl(B\cup\{e\})$ with 
$a\in \cl(b,e,d)$.

If $b\in \cl(B,e)$ or $a\in \cl(b,e)$ the result follows. It remains to consider the case where the set 
$\{a,b,e\}$
is independent.

Let $(a',b')\models \Lstp(a,b/B\cup\{e\})$ with $a'b'\ind_{B\cup\{e\}} ab$. Then by weak $1$-basedness
$ab\ind_{a'b'} B\cup\{e\}$. In particular, $a\in \cl(a'b'b)$.
Both $b$ and $e$ are generic over the set $\{a'b'\}$, so $Lstp(b/a'b')=Lstp(e/a'b')$. We can choose a countable model $N$ with 
$a'b'\in N$ such that $\{e,b\}\ind N$, and by uniqueness of generic types 
$tp(b/N)=tp(e/N)$.
Let $f$
be a partial automorphism taking $Nb$ to $Ne$ and extend it (using QM5) to a map sending $Nba$ to $Ned$ for some $d\in M$. Note that $Lstp(a,b/N)=Lstp(d,e/N)$
and that $ab\ind_{N}de$.
In particular, since $b\in \cl(aa'b')$, we also have $d\in \cl(ea'b')$.
Note that $\dim(ab/N)=\dim(ab/a'b')=1=\dim(ed/a'b')=\dim(ed/N)$ and thus $Lstp(ab/a'b')=Lstp(de/a'b')$ and $ab\ind_{a'b'}de$.

Using again weak $1$-basedness we have $a'b'\ind_{ab}de$, so we also get $d\in \cl(eab)$. 
Since $d\in \cl(eab)$ and $d\in \cl(ea'b')$, and by construction $a'b'\ind_{B\cup\{e\}} ab$, we also get $d\in \cl(eB)$. It remains to prove 
$a\in \cl(b,e,d)$.

Using weak 1-basedness one more time, $ab\ind_{de}a'b'$. This implies that $a\in \cl(b,e,d)$, which is what we wanted.
\end{proof}

Now we prove the converse of the previous result.

\begin{theo}\label{locmodimpliesw1b}
Assume $(M,\cl)$ is quasiminimal pregeometry structure and that for any
$e\in M$ generic (so $e\not \in cl(\emptyset)$) the localization $(M,cl_e)$ is a quasiminimal modular structure. Then $(M,\cl)$ is weakly $1$-based.    
\end{theo}

\begin{proof}
Let $A\subseteq M$ and let $\vec c,\vec d\in M^n$ be such that $\Lstp(\vec d/A)=\Lstp(\vec c/A)$ and $\vec d\ind_A \vec c$. We need to show $\vec c \ind_{\vec d} A$.
Choose $e\in M$ generic over $A,\vec c,\vec d$. Then $\vec d\ind_{A} e$, $\vec c\ind_{A} e$ and by stationarity $\Lstp(\vec d/eA)=\Lstp(\vec c/eA)$.
Choose a small model $N\subset M$ such that $Ae\subseteq N$ and $\vec c \vec d\ind_{Ae} N$. Using stationarity $\Lstp(\vec d/N)=\Lstp(\vec c/N)$. 
We also have $\Lstp(e\vec d/N)=\Lstp(e\vec c/N)$ as well as $\Lstp(e\vec d/\cl(eA))=\Lstp(e\vec c/\cl(eA))$. In particular, 
$\cl(eA)\cap \cl(e\vec c)=\cl(eA)\cap \cl(e\vec d)$. 

Since $(M,\cl_{e})$ is a modular structure, we have $e \vec c \ind_{\cl(Ae)\cap \cl(e\vec c)} Ae$ and by the previous conclusion
$e \vec c \ind_{\cl(Ae)\cap \cl(e\vec d)} Ae$. Since $e\vec c\ind_{Ae}e\vec d$ we get  $e \vec c \ind_{\cl(e\vec d)} Ae\vec d$. Since $e$ is generic over $A,\vec c,\vec d$ we get
 $\dim_{cl}(\vec c/e\vec d)=\dim_{\cl}(\vec c/\vec d)$ and thus $\vec c \ind_{\vec d} A$ as desired.
\end{proof}

The previous proof also works if we localize in any finite set. The proof is the same one as above. 

\begin{lem}\label{locmodsetimpliesw1b}
Assume $(M,\cl)$ is quasiminimal pregeometry structure and that for some finite set $B\subset M$ the localization $(M,cl_B)$ is a quasiminimal modular structure. Then $(M,\cl)$ is weakly $1$-based.  
\end{lem}

We obtain an analogous result to Lemma 4.2.4 \cite{Bu}, namely, a quasiminimal pregeometry structure is weakly $1$-based if and only if it is locally modular iff for some small set $B\subset M$ the localization $(M,cl_B)$ is a quasiminimal modular structure.

We will now give an example of a nontrivial linear quasiminimal pregeometry structure; the point of the example is to show that there are interesting objects that can not be interpreted as a strongly minimal theory but that can be studied as a quasiminimal pregeometry structure.

\begin{example}\label{vector-space} \rm 
Let $T=Th(V,+,0,\lambda\cdot)_{\lambda\in \mathbb{F}}$ be the theory of an infinite-dimensional vector space over a countably \emph{infinite} division ring $\mathbb{F}$. Let $(PG(V),\cl^*)$ be the associated projective geometry. It is known (see \cite{Pibook}, Ch.V, Proposition 2.2)  that there does not exist a strongly minimal structure whose natural pregeometry is a projective geometry over an infinite division ring.  However, as we show below,  we can represent  $(PG(V),\cl^*)$ as a quasiminimal pregeometry structure. For a nonzero $u\in V$, we denote by $u^*$ the line spanned by  $u$ as the element of $PG(V)$. And we write $u^*\in\cl^*(v_1^*,\ldots, v_n^*)$ whenever $u=\alpha_1 v_1+\dots +\alpha_n v_n$ for some $\alpha_i\in\mathbb{F}$. We will construct a natural language coming from the underlying vector space that makes an infinite dimensional projective geometry over any countable division ring into a quasiminimal structure.

For any $n\ge 1$ and $\vec\lambda=(\lambda_1,\ldots,\lambda_n)\in\mathbb{F}^n$ let the relation $R_{\vec \lambda}(x_1,\ldots,x_n)$ on $PG(V)$ be defined as follows: 

$R_{\vec \lambda}(u_1^*,\ldots,u_n^*)\iff$ for some $v_1,\ldots,v_n\in V$ such that $u_i^*=v_i^*$ for all $1\le i\le n$ we have $\lambda_1v_1+\ldots+\lambda_nv_n=0$. 

Note that for this particular representation, we can see that the structure is not strongly minimal. Indeed, for any $u^*\neq v^*$, $R_{(1,1,-1)}(u^*,v^*,x)$ defines the infinite co-infinite set $\cl^*(u^*,v^*)\backslash\{u^*,v^*\}$.

We claim that $(PG(V),\cl^*,R_{\vec\lambda})_{\vec\lambda\in \mathbb{F}^n,n\ge 1}$ is a linear quasiminimal pregeometry structure. Conditions (QM1-QM4) are clear and (QM5) over empty set is also clear, linearity follows from modularity of $\cl^*$.

We prove s stronger form of (QM5): we show homogeneity over a countable set $A$ which may not be closed.

By modularity, whenever $w^*\in\cl^*(u^*,v_1^*,\ldots, v_n^*)$, there exists $v^*\in\cl^*(v_1^*,\ldots,v_n^*)$ such that $w^*\in\cl^*(u^*,v^*)$. Thus, it suffices to show: 
whenever $A=(a_1^*,a_2^*,\ldots)$, $C=(c_1^*,c_2^*,\ldots)$ are two countable tuples of distinct elements in $PG(V)$ such that
$\qftp^*(A)=\qftp^*(C)$, and $b=\alpha _1a_1+\alpha_2a_2$, $b^*\not\in A$ (the case when $b^*\in A$ is trivial), there exists 
 $d^*\in\cl^*(c^*_1,c^*_2)$ such that $\qftp^*(Cd^*)=\qftp^*(Ab^*)$.   Note that $\alpha_1,\alpha_2\neq 0$. We will consider two cases. 

 \textbf{Case (1):}  $a_1^*\not\in\cl^*(a_2^*,a_3^*,\ldots)=\cl^*(A\backslash\{a_1^*\})$.
 Take any $d^*\in\cl^*(c_1^*,c_2^*)$ such that $d^*\neq c_1^*, c_2^*$.
 Then we claim that for any $n\geq 1$, $$\qftp^*(b^*,a_1^*, a_2^*,\ldots, a_n)=\qftp^*(d^*, c_1^*, c_2^*,\ldots, c_n^*).$$ Since $\qftp^*(a_1^*, a_2^*,\ldots, a_n)=\qftp^*(c_1^*, c_2^*,\ldots, c_n^*)$, replacing, if needed, the representatives $c_i$ of the classes $c_i^*$, we may assume $\qftp(a_1, a_2,\ldots, a_n)=\qftp(c_1,c_2,\ldots, c_n)$. Let $\beta_1,\beta_2\in\mathbb{F}\backslash\{0\}$  be such that  $d=\beta_1 c_1+\beta_2 c_2$. Note that since $\beta_1,\beta_2\neq 0$ we do have $d^*\in\cl^*(c_1^*,c_2^*)\setminus (\cl^*(c_1^*)\cup \cl^*(c_2^*))$. By assumption, $a_1^*\not\in\cl^*(a_2^*,\ldots, a_n^*)$, so $a_1\not\in\cl(a_2,\ldots, a_n)$. Let $a_1'=\frac{\alpha_1\beta_2}{\alpha_2\beta_1}a_1$. Since generic types in $V$ are unique we get $$\qftp(a_1',a_2,\ldots,a_n)=\qftp(a_1,a_2,\ldots, a_n)=\qftp(c_1,c_2,\ldots,c_n).$$
 
 Moreover, we have  $\qftp(\beta_1a'_1+\beta_2a_2,a_1',a_2,\ldots,a_n)=\qftp(\beta_1 c_1+\beta_2 c_2,c_1,c_2,\ldots,c_n)=\qftp(d,c_1,c_2,\ldots,c_n)$.

 Note that $\beta_1a_1'+\beta_2a_2=\frac{\beta_2}{\alpha_2}(\alpha_1a_1+\alpha_2a_2)=\frac{\beta_2}{\alpha_2}b$, so $(\beta_1a_1'+\beta_2a_2)^*= b^*$. This shows that for every $n\geq 1$ $$\qftp^*(b^*,a_1^*,a_2^*,\ldots, a_n^*)=\qftp^*(d^*,c_1^*,c_2^*,\ldots, c_n^*),$$ and hence, also $\qftp^*(Ab^*)=\qftp^*(Cd^*)$.

 \textbf{Case (2)}: $a_1^*\in\cl^*(a_2^*,a_3^*,\ldots)=\cl^*(A\backslash\{a_1^*\})$.
Reordering $A$ if needed, we may assume that for some $m>2$, $\{a_1^*,a_2^*,a_3^*,\dots, a_m^*\}$ is a minimal dependent set (a {\it circuit}, in the matroid-theoretic terminology, or a $\cl^*$-$m$-gon in the sense of \cite[Defn 3.20]{BeVa}). Thus, the representatives $a_1,\dots,a_{m-1}$ of the classes $a_1^*,\ldots,a_{m-1}^*$ are linearly independent  and $a_m=\beta_1a_1+\ldots +\beta_{m-1}a_{m-1}$ for some $\beta_i\in\mathbb{F}\backslash\{0\}$. Since $\qftp^*(A)=\qftp^*(C)$, by choosing the appropriate representatives for $c_1^*,\dots,c_m^*$ we may assume that $\qftp(a_1,\ldots,a_m)=\qftp(c_1,\ldots, c_m)$. 

Thus, we also have $c_m=\beta_1c_1+\ldots +\beta_{m-1}c_{m-1}$. Let $d=\alpha _1c_1+\alpha_2c_2$ then we have $\qftp(b,a_1, a_2,\ldots, a_m)=\qftp(d, c_1, c_2,\ldots, c_m)$.

\textbf{Claim} For any $n>m$, $\qftp^*(b^*,a_1^*, a_2^*,\ldots, a_n^*)=\qftp^*(d^*, c_1^*, c_2^*,\ldots, c_n^*)$. 

First, since $\qftp^*(a_1^*, a_2^*,\ldots, a_n^*)=\qftp^*( c_1^*, c_2^*,\ldots, c_n^*)$, we have $\qftp(a_1,\ldots, a_n)=\qftp(c_1',\ldots,c_n')$ 
where $c_i'=\gamma_i c_i$ for some nonzero $\gamma_i\in \mathbb{F}$. 

Then $c_m'=\beta_1c_1'+\ldots +\beta_{m-1}c_{m-1}'$. Let $d'=\alpha _1c_1'+\alpha_2c_2'$.
Then $$\qftp(b,a_1,\ldots,a_n)=\qftp(d',c_1',\ldots,c_n').$$
It suffices to show then that $d'^*=d^*$. We have:

$$c_m'=\gamma_m c_m=\sum_{i=1}^{m-1}\beta_ic_i'=\sum_{i=1}^{m-1}\beta_i\gamma_i c_i.$$

Then $$\gamma_m \sum_{i=1}^{m-1}\beta_ic_i=\sum_{i=1}^{m-1}\beta_i\gamma_i c_i.$$ By the independence of $c_1,\dots,c_{m-1}$, it follows that $\gamma_i=\gamma_m$ for all $1\leq i\leq m-1$. Thus, using the fact that $d=\alpha _1c_1+\alpha_2c_2$ we get $d'=\alpha_1 c_1'+\alpha_2 c_2'=\alpha_1 \gamma_m c_1+\alpha_2\gamma_m c_2=\gamma_m d$, and therefore, $d'^*=d^*$, as we wanted.

\end{example}

Let us finish this section with some references in this subject. Locally modular non-trivial quasiminimal pregeometry structures have been studied in \cite{HK,Ka21}. For example, a group can be interpreted in $(M,\cl)$ (see Remark 3.24 \cite{HK} and Theorem 3.9, Lemma 3.23 \cite{Ka21}) using a generalization of imaginary sorts by adding classes for all invariant equivalence relations. In this paper we will take a different approach and avoid the use of such sorts, precisely the contributions of this section were new proofs that avoid the use of generalized imaginaries.

\section{H-structures and beautiful pairs of quasiminimal classes}\label{sec:dense-codense}
In this section we introduce two  dense-codense expansions of a quasiminimal pregeometry structure: $H$-structures and beautiful pairs. We show each of these expanions can be axiomatized with a single $\mathcal{L}_{\omega_1 \omega}(Q)$-sentence and that the resulting class is $\omega$-stable. We start by defining dense-codense expansions, similar to how it is done in the geometric setting \cite{BeVa, BeVaLP}.

\begin{defn}
    
Let $P$ be a new unary predicate that does not belong to $\mathcal{L}$, let $M\in \mathcal{C}$ and let $\mathcal{L}_P=\mathcal{L}\cup\{P\}$. We say that the expansion $(M, P)$ is \emph{dense-codense} if:\\
(D1) (Density/coheir property) If $b_1,\dots,b_n \in M$ and $q \in S_1(\cl(b_1,\dots,b_n))$ is the unique
non-algebraic type, there is $a \in P(M)$ such that $a\models q$.\\
(D2) (Extension/codensity property) If $b_1,\dots,b_n \in M$ and $q \in S_1(\cl(b_1,\dots,b_n))$ is the unique
non-algebraic type, there is $a \in M$, $a\models q$ and $a\not \in  \cl(b_1,\dots,b_n \cup P(M ))$.
\end{defn}

\begin{obse}
    Note that the expansion $(M, P)$ is \emph{dense-codense} if
    and only if $\dim(P(M))=\infty$ (this is the density property) and $\dim(M/P(M))=\infty$ (this is the codensity property)
\end{obse}

Note that axiom (D1) implies that if $(M,P)$ is a dense-codense pair, then $M$ satisfies axiom (QM2) and $M$ is a quasiminimal pregeometry structure. It is also important to point out that in first order geometric theories the geometry is controlled definably using elimination of the quantifier $\exists^{\infty}$. The first order arguments (see \cite{BeVa,BeVaLP}) about dense-codense expansions that use elimination of $\exists^{\infty}$ 
will be replaced in quasiminimal classes by arguments that involve Fact \ref{definability-density}. The next proposition is a good example of this idea:

\begin{prop}\label{density-codensity}
Let $\mathcal{C}$ be a quasiminimal class. Then there is an $\mathcal{L_P}_{\omega_1 \omega}$-sentence $\phi_{\mathcal{C},P}$ such that for any $M\in \mathcal{C}$, $(M,P)\models \phi_{\mathcal{C},P}$ if and only if $(M,P)$ is a dense-codense pair.
\end{prop}

\begin{proof}
 By Fact \ref{definability-density},
for each $n \in \mathbb{N}$ there is a $\mathcal{L}_{\omega_1,\omega}$-formula
$\pi_n(x, y_1,\dots,y_n)$ such that for all $M \in \mathcal{C}$ and all $a, b_1,\dots, b_n \in M$, we
have $a \in \cl_M(b_1,\dots,b_n)$ iff $M \models \pi_n(a,b_1,\dots,b_n)$. 

The density property is axiomatized through the sentence:

$$\bigwedge_{n\geq 1} \Big{(}\forall b_1 \dots \forall b_n \exists a\in P 
\neg \pi_n(a,b_1,\dots,b_n)\Big{)}$$

and the codensity property through the sentence:

$$\bigwedge_{n\geq 1}  \Big{(}\forall b_1 \dots \forall b_n  \exists a \Big{[}\bigwedge_{m\geq 1}\forall c_1\in P \dots \forall c_m\in P
\neg \pi_{n+m}(a,b_1,\dots,b_n,c_1,\dots,c_m)
\Big{]}\Big{)}$$

both sentences and their conjunction $\phi_{\mathcal{C},P}$ are expressible in the language ${\mathcal{L}_P}_{\omega_1 \omega}$ as desired. 
\end{proof}

As we noted earlier, the class of dense-codense expansions only includes those members of $\mathcal{C}$ with infinite dimension with respect to the closure operator $\cl$. This is analogous to what happens in the first order geometric case, the axioms presented in \cite{BeVa,BeVaLP} imply that if $(M,P)$ is an $H$-structure (resp. lovely pair) of geometric theories, then $M$ is $\aleph_0$-saturated.

\begin{cor}\label{density-codensitygeneral}
The class of dense-codense expansions of $\mathcal{C}$ is axiomatizable by a single $\mathcal{L}_{\omega_1 \omega}(Q)$-sentence. 
 \end{cor}

\begin{proof}
By Fact \ref{Defclass} there is a single $\mathcal{L}_{\omega_1 \omega}(Q)$-sentence $\psi_{\mathcal{C}}$ whose models is the class $\mathcal{C}$. By Proposition \ref{density-codensity} there is another $\mathcal{L}_{\omega_1 \omega}$-sentence $\phi_{\mathcal{C},P}$ saying that $P$ is dense/codense. 
\end{proof}

The expansions by dense-codense predicates do not define a complete theory. For example, when defining an $H$-structure, we also require that the elements in $P(M)$ are independent in the sense of the pregeometry (see \cite{BeVa}); when defining a beautiful pair, we require that $P(M)$ is a member of $\mathcal{C}$.

\begin{defn}
Let $M\in \mathcal{C}$ and let $(M, P)$ be a dense-codense expansion. We say that $(M, P)$ is an \emph{$H$-structure} if, in addition, we have the following property:

(DH) ($P$ is cl-independent) If $b_1,\dots,b_n \in P(M)$ are different, then they are $cl$-independent.

We say that $(M, P)$ is a \emph{beautiful pair} if we also have the following property:

(DP) $P(M)\in \mathcal{C}$ and $P(M)\preceq M$.
\end{defn}

\begin{rema}\label{closure-of-H} Let $M\in \mathcal{C}$ and consider an $H$-structure $(M,H)$.
Then $P(M)=\cl(H(M))$ is an elementary substructure of $M$ which is still dense-codense, and, therefore, $(M,P)$ is a beautiful pair. Additionally, $(M,P)$ can be viewed as a $\mathcal{L}_{\omega_1\omega}$-reduct of $(M,H)$.
    
\end{rema}

\begin{prop}\label{H-structures-fo}
If $\mathcal{C}$ is a quasiminimal class, then the class of $H$-structures is also axiomatizable in the language ${\mathcal{L}_P}_{\omega_1 \omega}(Q)$ and so is the class of beautiful pairs.
\end{prop}

\begin{proof} By Corollary \ref{density-codensitygeneral} there is a single $\mathcal{L_P}_{\omega_1 \omega}(Q)$-sentence $\varphi$ whose models form the class of dense-codense expansions associated to $\mathcal{C}$. Such a pair is an $H$-structure if it satisfies the additional ${\mathcal{L}_P}_{\omega_1 \omega}$-sentence: 
$$\bigwedge_{n\geq 1}\big{(}\forall b_1\in P\dots\forall b_{n+1} \in P \big{[}\bigwedge_{i<j\leq n+1}b_i\neq b_j \implies \neg \pi_{n}(b_1,\dots,b_n,b_{n+1})\big{]}\big{)}$$

Similarly, we can say that a dense-codense expansion is a beautiful pair by first adding to $\varphi$ the ${\mathcal{L}_P}_{\omega_1 \omega}(Q)$-sentence stating that $P(M)\models \psi_{\mathcal{C}}$.
We can say $P(M)$ is a \emph{closed subset} of $M$ using the ${\mathcal{L}_P}_{\omega_1 \omega}$-sentence $\varphi_{c}$: 
$$\bigwedge_{n\geq 1}\big{(}\forall b_1 \dots \forall b_n \in P \forall x \big{[}\pi_{n}(x,b_1,\dots,b_n) \implies x\in P \big{]}\big{)}$$
The sentence $\varphi_c$ says that taking the closure in the sense of $M$ of a $n$-tuple from $P(M)$ does not add elements outside $P(M)$.

\textbf{Claim} Assume $(M,P)$ is a dense-codense pair satisfying that $P(M)\in \mathcal{C}$ and the sentence $\varphi_c$. Then the embedding of $P(M)$ inside $M$ is a \emph{closed embedding} and thus $P(M)\preceq M$.

Let $A\subset P(M)$ and assume that $A$ is closed inside the structure $P(M)$, that is $A=\cl_{P(M)}(A)$.
Let $b\in \cl_M(A)$ with $b\in M$, then $b\in \cl(a_1,\dots,a_n)$ for some $a_1,\dots,a_n\in A$. Then by Fact \ref{definability-density} we have $M\models \pi_{n}(b,a_1,\dots,a_n)$ and $\pi_n(y,x_1,\dots,x_n)$ is quantifier free. Since $P(M)\models \varphi_c$ we have $b\in P(M)$. Since $P(M)$ is substructure of $M$ and $\pi_n(y,x_1,\dots,x_n)$ is quantifier free, we have $P(M)\models \pi_{n}(b,a_1,\dots,a_n)$ and since $P(M)$ belongs to $\mathcal{C}$ by Fact  \ref{definability-density}, we also have $b\in \cl_{P(M)}(A)=A$.

 \end{proof}

 Our goal is to study tameness properties in the class of $H$-structures and in the class of beautiful pairs associated to $\mathcal{C}$. In particular, we want to understand reasonable analogues to properties that hold for $H$-structures (respectively beautiful pairs) associated to first order geometric structures \cite{BeVa, BeVaLP}, like characterizing types in the extended language and proving that the expansion is stable. Before we continue, we need to fix some notation for the rest of this section:

 \begin{nota}

 If $(M,H)$ is an $H$-structure associated to $\mathcal{C}$, we write $\tp_H(\vec a)$ for its  $\mathcal{L_H}_{\omega_1\omega}$-type in the extended language. 

 
 Similarly, if $(M,P)$ is a beautiful pair associated to $\mathcal{C}$, we write $\tp_P(\vec a)$ for its  $\mathcal{L_P}_{\omega_1\omega}$-type in the extended language. 


 If $B\subseteq M$ and $(M,H)$ is a dense-codense expansion,
 we write $H(B)$ for $B\cap H(M)$. If $(M,H)$ is a dense-codense expansion, we say $\vec a\in M^n$ is \emph{$H$-independent} if $\dim_{\cl}(\vec a/H(\vec a))=\dim_{\cl}(\vec a/H(M))$.  Furthermore, if $C\subseteq M$ and $C$ is $H$-independent, we say $\vec a\in M^n$ is \emph{$H$-independent over $C$} if $\dim_{\cl}(\vec a/CH(\vec a))=\dim_{\cl}(\vec a/CH(M))$. 
 \end{nota}
 
We first study types and some homogeneity properties in the expansions. The argument is written for $H$-structures, but the same argument holds for beautiful pairs.
  The result is an analogue of Lemma 2.12 \cite{BeVa}, it also shows that our structures are $\aleph_0$-homogeneous, a reasonable analogy to axioms QM4 and QM5 above.

\begin{prop}\label{H-typesomega1omega}
Let $(M,H)$ and $(N,H)$ be $H$-structures associated to $\mathcal{C}$ and let 
$\vec a\in M$, $\vec a'\in N$ be $H$-independent finite tuples such that
$\tp(\vec a, H(\vec a))=\tp(\vec a', H(\vec a'))$.  Then
$tp_{H}(\vec a)=\tp_{H}(\vec a')$. Furthermore, for any $b\in M$ there is $b' \in N$ such that 
$tp_{H}(\vec ab)=\tp_{H}(\vec a'b')$.  
 \end{prop}

\begin{proof}
 Write $\vec a=\vec a_0 \vec a_1 \vec h$, where $\vec a_0$ is
independent over $H(M)$, $\vec h=H(\vec a)\in H(M)$ and $\vec a_1\in cl (\vec
a_0\vec h)$. Similarly write $\vec a'=\vec a_0' \vec a_1'\vec h'$.

To prove both assumptions, it suffices to show that for any $b\in M$ there are $\vec h_1\in H(M)$, $\vec h_1'\in H(N)$ and $b'\in N$ such that $\vec a\vec h_1
b$ and $\vec a'\vec h_1' b'$ are each $H$-independent, $\tp(\vec
a_0 \vec a_1 \vec h \vec h_1b)=\tp(\vec a_0' \vec a_1' \vec h'
\vec h_1'b')$, and $b\in H(M)$ iff $b'\in H(N)$.

\noindent

Case 1: $b\in \cl(\vec a)$. There is an $\mathcal{L}_{\omega_1 \omega}$-elementary partial map from $M$ to $N$ sending $\vec a$ to $\vec a'$ that can be extended (using axiom QM5) to a map from 
$\vec ab$ to $\vec a'b'$. Clearly $\tp(\vec ab)=\tp(\vec a'b')$ and both tuples remain $H$-independent. Assume now that $b\in H$. Since $\vec a$ is $H$-independent we must have $b\in \vec a$ and thus $b'\in \vec a'$, finally $\tp(\vec a, H(\vec a))=\tp(\vec a', H(\vec a'))$ implies that $b'\in H$.

Case 2: $b\not \in \cl(\vec a)$ and $b\in H$. Then $b$ belongs to $H$ and satisfies the unique generic $\mathcal{L}_{\omega_1\omega}$-type over $\vec a$. Since $(N,H)$ is an $H$-structure, it satisfies the density property and there is $b'\in H(N)$ 
such that $N\models \neg \pi_n(b',\vec a')$ and thus the map sending $\vec a b$ to $\vec a'b'$ is an $\mathcal{L}_{\omega_1 \omega}$-elementary map and both tuples remain $H$-independent.

Case 3: $b\in \cl(\vec ah_1\dots h_k)$ for some $h_1\dots h_k\in H(M)$. We may assume that $h_1\dots h_k$ are $\cl$-independent over $\vec a$. Apply first case 2 to find $h_1'\dots h_k'\in H(M)$ with 
$\tp(\vec a,h_1\dots h_k)=\tp(\vec a', h_1'\dots h_k')$
and note that the tuples remain $H$-independent. Now find the desired $b'$ applying case 1. 

Case 4: $b\not \in \cl(\vec ah_1\dots h_k)$ for all $h_1\dots h_k\in H(M)$ and all $k\geq 1$. Since $(N,H)$ is an $H$-structure, it satisfies the codensity property so there is $b'\in N\setminus \cl(\vec a'H(N))$. The map sending $\vec a b$ to $\vec a'b'$ is an $\mathcal{L}_{\omega_1 \omega}$-elementary map and the tuples remain $H$-independent.
\end{proof}

\begin{cor}\label{completasstheory}
    Any two $H$-structures (resp. beautiful pairs) associated to the same class $\mathcal{C}$ have the same $\mathcal{L}_{\omega_1 \omega}$-theory.
\end{cor}

\begin{proof}
The uniqueness of the $\mathcal{L}_{\omega_1 \omega}$-theory follows from Proposition \ref{H-typesomega1omega}. Furthermore, by Proposition \ref{H-structures-fo} the class can be described by a single $\mathcal{L}_{\omega_1 \omega}(Q)$-sentence.
\end{proof}

\begin{cor}
Recall that we assumed in the beginning of the section that the class $\mathcal{C}$ has quantifier elimination, meaning the quantifier free $\mathcal{L}_{\omega_1\omega}$-formulas isolate the type (part of the assumptions from \cite{BHHKK}). It follows that the type of an $H$-independent tuple is determined by its quantifier free $\mathcal{L}_{\omega_1\omega}$-formulas and selecting the collection of elements that belong to $H$. But, in general, we will require the use of existential formulas (and their negations) to extend a tuple to an $H$-independent tuple to determine the type. See example \ref{equiv-classes} below
for an explicit construction that shows the expansion does not eliminate quantifiers.

\end{cor}

\begin{obse}
By Proposition \ref{H-typesomega1omega}, the syntactic $\mathcal{L}_H$-types in the new language isolate the Galois type in the class of $H$-structures (resp. beautiful pairs) if we work in a sufficiently "saturated" model.
\end{obse}

\begin{ques}\label{omega1omegQcase}
Proposition \ref{H-typesomega1omega} provides a back and forth argument that shows how to determine the $\mathcal{L}_{H\omega_1\omega}$-type in the expansions. Since both the original class and the associated class of $H$-structures are described by $\mathcal{L}_{\omega_1 \omega}(Q)$-sentences, it could be a natural question to study the $\mathcal{L}_{H \omega_1 \omega}(Q)$-types using the tools from \cite{Ca}. More specifically: let $(M,H)$ and $(N,H)$ be $H$-structures associated to $\mathcal{C}$ and let 
$\vec a\in M$, $\vec a'\in N$ be $H$-independent finite tuples such that
$\tp(\vec a, H(\vec a))=\tp(\vec a', H(\vec a'))$. Assume furthermore that the structures involved have large dimensions, that is, $\dim(H(M))\geq 2^{\aleph_0}$, $\dim(M/H(M))\geq 2^{\aleph_0}$ and $\dim(N(M))\geq 2^{\aleph_0}$, $\dim(N/H(N))\geq 2^{\aleph_0}$. Does it follow that 
$tp_{\mathcal{L}_{H\omega_1\omega}(Q)}(\vec a)=\tp_{\mathcal{L}_{H\omega_1\omega}(Q)}(\vec a')$?

Note that, using the $Q$ quantifier, we can work with an $\mathcal{L}_{H\omega_1\omega}(Q)$ axiomatizable subclass of pairs $(M,H)$ where $H(M)$ has dimension bigger than $\aleph_0$ by adding the sentence $Q(H)$ to the axioms. Is there an analogous $\mathcal{L}_{H\omega_1\omega}(Q)$sentence that guarantees that $\dim(M/H(M))>\aleph_0$?
\end{ques}

In order to reproduce the ideas from \cite{BeVa} in the setting of $H$-structures we will also need the notion of an $H$-basis. We provide the definition, the proof of existence and uniqueness is essentially the same one given in \cite{BeVa}.

\begin{defn} Let $(M,H)$ be an $H$-structure, let $B \subseteq M$ be closed and $H$-independent and let $\vec a \in M$. We write $HB(\vec a/B)$ for the smallest finite subset $H_0\subset H(M)$ such that $\dim(\vec a/BH(M))=\dim(\vec a/BH_0)$ and we call it the \emph{$H$-basis of $\vec a$ over $B$}.
For a general set $A$, we write $HB(A)$ for the smallest set $H_0\subseteq H$ (smallest under containment) such that $\dim(\vec a/H_0)=\dim(\vec a/H(M))$ for all finite tuples $\vec a\in A$.
In general, if $B$ is not closed or $H$-independent and $\vec a \in M$, we write $HB(\vec a/B)$ for $HB(\vec a/\cl(B\cup HB(B)))$. 
\end{defn}

An $H$-basis of a tuple $\vec a$ over a closed set $B$ provides a "projection" of the tuple into the smallest part of $H$ that interacts with the tuple. A key property of H-basis is additivity:

\begin{prop}\label{H-basis-additivity}
Let $(M,H)$ be an $H$-structure, let $C \subseteq M$ be closed and $H$-independent and let $\vec a,\vec b \in M$. Then
$$HB(\vec a\vec b/C)=HB(\vec a/C)\cup HB(\vec b/C\vec a)$$
Furthermore, the union above is a \emph{disjoint} union.
\end{prop}

\begin{proof} We start by proving the equality, we first show $HB(\vec a/C)\cup HB(\vec b/\vec a C)\subseteq HB(\vec a\vec b/C)$.
Let $H_0=HB(\vec a\vec b/C)$. By definition we have $\vec a \vec b\ind_{CH_0} H(M)$. Then we get $\vec a \ind_{CH_0} H(M)$ and $\vec b\ind_{CH_0\vec a} H(M)$. From the last part we get $\vec b\ind_{CH_0\vec aHB(\vec a/C)} H(M)$. By minimality of $H$-basis we get $HB(\vec a/C)\subseteq H_0$ and $HB(\vec b/\vec a C)\subseteq H_0$ and thus $HB(\vec a/C)\cup HB(\vec b/\vec a C)\subseteq H_0$.

To prove the other containment, let $H_1=HB(\vec a/C)$ and $H_2=HB(\vec b/CH_1\vec a)$. As before, by definition we have $\vec a \ind_{CH_1} H(M)$ and 
$\vec b \ind_{C\vec a H_1\cup H_2} H(M)$. Then we get $\vec a \ind_{CH_1\cup H_2} H(M)$ and thus
$\vec a\vec b \ind_{CH_1\cup H_2} H(M)$. By minimality of $H$-basis we get $H_1\cup H_2\supseteq H_0$.

Now we show that the union is disjoint. By definition, we have $HB(\vec b/C\vec a)=HB(\vec b/C\vec aHB(C\vec a))$ and by additivity $HB(C\vec a)=H(C)\cup HB(\vec a/C)$. We claim that by \emph{minimality} of an $H$-basis, $HB(\vec b/C\vec a)$ must be disjoint from $HB(\vec a/C)\cup H(C)$.

Assume not and let $H_3=HB(C\vec a)\cap HB(\vec b/C\vec aHB(C\vec a))$ and define $H_2'=HB(\vec b/C\vec aHB(C\vec a))\setminus H_3=H_2\setminus H_3$. Then $H_2'\subsetneq HB(\vec b/C\vec aHB(C\vec a))=H_2$ and since, as sets, we have $C\vec aHB(\vec a/C)H_2'=C\vec aHB(\vec a/C)H_2$, we get $\vec b\ind_{C\vec aHB(\vec a/C)H_2'} H(M)$, a contradiction to the minimality of $H_2$.
\end{proof}


\begin{prop}\label{H-basis-definability}
Let $(M,H)$ be an $H$-structure, let $B \subseteq M$ be closed and $H$-independent and let $\vec a \in M$.
Then there is an $\mathcal{L}_{\omega_1\omega}$-formula $\varphi(\vec y)$ with parameters in $B\vec a$ such that an enumeration of $HB(\vec a/B)$ realizes $\varphi(\vec y)$ and the other realizations of the formula are permutations of $HB(\vec a/B)$.
\end{prop}

\begin{proof}
Write $\vec a=\vec a_1\vec a_2$ with $\vec a_1$ $\cl$-independent from $BH(M)$ and $\vec a_2\in \cl(\vec bH_0\vec a_1)$, where $H_0=HB(\vec a/B)$ and $\vec b$ is a finite tuple in $B$. Let $\vec h=(h_1,\dots,h_k)$ be an enumeration of $H_0$ and let $\vec y$ a collection of $k$ variables. Let $s_1=|\vec a_1|$, $s_2=|\vec b|$ and $s=s_1+s_2+k$. We may also enumerate $\vec a_2=(a_{21},\dots,a_{2r})$ 
and notice that for each $j\leq r$, $a_{2j}\in \cl(\vec a_1,\vec b,H_0)$, so $\pi_{s}(a_{2j},\vec a_1,\vec b,\vec h)$ holds. Now let $\varphi(\vec y)$ be the formula:

$$(\bigwedge_{i<j\leq k} y_i\neq y_j) \wedge \bigwedge_{i\leq k} (y_i\in H) \wedge \bigwedge_{j\leq r} \pi_{s}(a_{2j},\vec a_1,\vec b,\vec y)$$

The tuple $\vec h=(h_1,\dots,h_k)$ satisfies the formula: the elements are all different, they belong to $H(M)$ and $\vec a_2\in \cl(\vec a_1,\vec b,\vec h)$.

Now assume $\vec h'$ is another realization of $\varphi(\vec y)$. Then the elements of $\vec h'$ are different, they belong to $H(M)$ and $\vec a_2\in \cl(\vec a_1,\vec b,\vec h')$. By minimality of $H$-basis, we have that the elements in $H_0$ belong to $\vec h'$ and since $|\vec h|=|\vec h'|$, one tuple is a permutation of the other.
\end{proof}

Let us consider some examples of quasiminimal classes from \cite[Examples 1.4]{Ki} and characterize the corresponding $H$-structures and beautiful pairs. We will also study $H$-basis in these examples. We start with an example of a disintegrated geometry:

\begin{example}\label{equiv-classes}
 Let $\mathcal{L}=\{E\}$ and consider the class $\mathcal{C}$ whose objects consist of infinite sets and $E$ is interpreted as an equivalence relation with infinitely many classes, all of whose blocks have size $\aleph_0$. The closure of a subset $A$ of a model $M$ in $\mathcal{C}$ is the union of the blocks meeting $A$. Given $\vec a\in M$,
 the generic $1$-type over $\vec a$ corresponds to the type of an element that does not belong to a block that meets $\vec a$. The density property says that $H(M)$ meets infinitely many blocks, the codensity property says that $H(M)$ does not meet infinitely many blocks. Being an $H$-structure says that $H(M)$ is dense-codense and that $H(M)$ intersects each block in at most a single point. A beautiful pair is a dense-codense expansion such that $P(M)$ contains each block that it intersects. Note that if $P(M)$ only contains infinitely many points from each block that it intersects
(instead of containing all points), then $P(M)$ is just a \emph{substructure} of $M$ instead of a \emph{closed substructure}.

 Let us consider again $H$-structures. Given $\vec a\in M$ write $\vec a=\vec a_1\vec a_2$ so that the blocks meeting $\vec a_1$ do not intersect $H(M)$ and the blocks meeting $\vec a_2$ intersect $H(M)$. For each of the blocks that intersect $\vec a_2$, let $\vec h\in H(M)$ be the unique elements in $H(M)$ belonging to the blocks. Then $\vec a_2\in \cl(\vec h)$ and $\vec h=HB(\vec a)$. Note that if $b_1,b_2\in M$ are such that $b_1,b_2\not \in H(M)$, the block of $b_1$ intersects $H(M)$ and the block of $b_2$ does not intersect $H(M)$,   then $qftp_H(b_1)=qftp_H(b_2)$ but 
 $tp_H(b_1)\neq tp_H(b_2)$ as they satisfy different existential formulas.
\end{example}


Now we consider an example with an underlying strongly minimal theory, where both $\cl$ and $H$ can be recovered from the associated algebraic closure and the corresponding theory of $H$-structures.

\begin{example}\label{ACF-Naturals}
 Let $\mathcal{L}_{rings}\cup \{Z\}$ be the language of rings together with a new unary predicate. Consider the class $\mathcal{C}$ whose objects consist of 
 algebraically closed fields of characteristic zero and the predicate $Z$ is interpreted as the integers. The operator $\cl$ is interpreted as field-theoretic algebraic closure. Since $ACF_0$ is strongly minimal, it is a geometric first order theory and we can consider the associated theory $T^{ind}$ of $H$-structures \cite{BeVa} and the associated theory of beautiful pairs $T^P$. For each $M\models ACF_0$ with infinite dimension, we build an $H$-structure by interpreting $H(M)$ as a infinite collection of independent transcendentals such that $dim(M/H(M))=\infty$. Similarly, we build a beautiful pair by interpreting $P(M)=acl(H(M))$ where $H(M)$ is as above.
 We can view each of these $H$-structures 
 as associated structures of the class $\mathcal{C}$ by interpreting $Z$ as the integers, similarly, we get a beautiful pair associated structures of the class $\mathcal{C}$ by interpreting $Z$ in each model of the beautiful pair as the integers. For each such an $H$-structure, the notion of $H$-basis is the same from both perspectives.
\end{example}

\begin{prop}\label{H_independentsubstructures}
Let $(M,H)$ be an $H$-structure (resp. beautiful pair) and let $(N,H) \preceq_{\mathcal{L}_{H_{\omega_1, \omega}}} (M,H)$ be an ${\mathcal{L}_H}_{\omega_1, \omega}$-substructure. Then $N$ is independent from $H(M)$ over $H(N)$.
\end{prop}

\begin{proof}
Let $(N,H) \preceq_{\mathcal{L}_{H_{\omega_1, \omega}}} (M,H)$ be an ${\mathcal{L}_H}_{\omega_1, \omega}$-substructure. 

Let $\vec a\in N$ and note that we have $\dim(\vec a/H(M))\leq \dim(\vec a/H(N))$. It remains to prove the other inequality. Let $\vec h\in H(M)$ be a finite tuple such that $\dim(\vec a/H(M))=\dim(\vec a/\vec h)$. We may write $\vec a=\vec a_1\vec a_2$, with $\vec a_1$ independent from $H(M)$ and 
$\vec a_2\in \cl_M(\vec a_1\vec h)$. Then for each $c\in \vec a_2$ we have $(M,H)\models \pi_n(c,\vec a_1\vec h)$ with $n=|\vec a_1 \vec h|$. As $(N,H) \preceq_{\mathcal{L}_{H_{\omega_1, \omega}}} (M,H)$, we have
$(N,H)\models \exists \vec h'\in H \wedge_{c\in \vec a_2} \pi_n(c,\vec a_1\vec h')$, and this gives the desired result.
\end{proof}

\begin{nota}
From now on, whenever $(N,H)\subseteq (M,H)$ are $H$-structures (respectively beautiful pairs) we write $(N,H) \preceq (M,H)$ whenever $(N,H) \preceq_{\mathcal{L}_{H_{\omega_1, \omega}}}(M,H)$.
\end{nota}

The converse of Proposition \ref{H_independentsubstructures} provides a characterization of when we have $(N,H) \preceq(M,H)$.

\begin{prop}\label{H_independentsubstructurescharomega1}
Let $(M,H)$ be an $H$-structure (resp. beautiful pair) and let $(N,H) \subseteq (M,H)$ be an $H$-substructure. Assume that $N$ is closed in $M$ (so $N\preceq M$) and that $N\ind_{H(N)} H(M)$.
Then $(N,H) \preceq(M,H)$.
\end{prop}

\begin{proof}
Let $\vec a\in N$ be a finite tuple and assume $\vec a$ is $H$-independent in $N$. So we may write $\vec a=\vec a_1 \vec a_2 \vec a_3$ with $\vec a_1$ an independent tuple independent from $H(N)$, $\vec a_2\in H(N)$ are independent elements and $\vec a_3\in \cl_N(\vec a_1\vec a_2)$. Since $N\preceq M$, the tuple $\vec a_1 \vec a_2$ remains independent in $M$ and $\vec a_3\in \cl_M(\vec a_1\vec a_2)$.

\textbf{Claim} The tuple $\vec a$ is $H$-independent in the structure $M$.

Since $N\ind_{H(N)} H(M)$ we get $\vec a\ind_{H(N)} H(M)$. We also have $\vec a\ind_{\vec a_2}H(N)$ so can conclude $\vec a\ind_{\vec a_2}H(M)$.

Also, since $(N,H) \subseteq (M,H)$ we have $H(M)\cap \vec a =H(N)\cap \vec a$. Thus the tuple $\vec a$ is $H$-independent in both structures and $\tp_N(\vec a,H(\vec a))=\tp_M(\vec a,H(\vec a))$, so following the proof of Proposition \ref{H-typesomega1omega} we get $(N,H,\vec a) \preceq_{\mathcal{L}_{H_{\omega_1, \omega}}}(M,H,\vec a)$.

\end{proof}


\begin{ques}
Recall Question \ref{omega1omegQcase}, where we asked what happened if types in the class of pairs where $\mathcal{L}_{\omega_1 \omega}(Q)$-types, instead of being $\mathcal{L}_{\omega_1 \omega}$. A related question was also posed by Jonathan Kirby to the authors: which class of pairs does one obtain if pair embeddings are $\mathcal{L}_{\omega_1 \omega}(Q)$-embeddings?   
\end{ques}

We continue now by counting types in the expansion. We will prove the class of $H$-structures (resp. beautiful pairs) associated to $\mathcal{C}$ is $\omega$-stable.

\begin{prop}\label{stability}
Let $(M,H)$ be an $H$-structure (resp. beautiful pairs), let $B \subseteq M$ be closed, $H$-independent and countable.
Then there is at most countably many types over $B$ realized in $(M,H)$.
\end{prop}

\begin{proof} We do the proof for $H$-structures, it is easy to modify the proof to deal with beautiful pairs.
Since $B$ is countable, there are countably many types realized in $B$. Now let $a\not \in B$. First consider $\tp(a/B)$, the type in the language $\mathcal{L}_{\omega_1 \omega}$. Since $\mathcal{C}$ is quasiminimal and $B=cl(B)$, there is a single choice for this generic type. Now we apply Proposition \ref{H-typesomega1omega}
and consider different cases:

If we have $a\in H(M)$, then $\tp_H(a/B)$ is determined by the fact that $a$ is generic over $B$ and $a\in H(M)$. There is a single type with this property.

Now assume $a\in \cl(H(M)B)$ and let $h_1,\dots,h_n\in H(M)$ be a minimal set (the $H$-basis of $a$ over $B$) such that $a\in \cl(Bh_1,\dots,h_n)$. For each $n$, there is a single $n$-type of the form $\tp(h_1,\dots,h_n/B)$
and since the closure of a countable set is countable (see axiom $QM3$), there are countably many options for the type $\tp(b/Bh_1,\dots,h_n)$. Thus there are countably many options for $\tp_H(b/B)$.

Finally, if we have $a\not \in \cl(BH(M))$, then $\tp_H(a/B)$ is determined by the fact that $a$ is generic over $B$ and $a$ is $\cl$-independent from $H(M)B$. There is a single type with this property.

This proves that there are at most countably many types over $B$ and the extended class of $H$-structures is $\omega$-stable.
\end{proof}

Now we want to show that the expansion preserves some reasonable analogue of superstability. Recall that for first order theories, an expansion to an $H$-structure preserves the stability spectrum and in particular preserves superstability \cite{BeVa}. In the current setting we need the notion of splitting and Lascar splitting.
As we did earlier, we will assume the reader is familiar with Lascar strong types (see for example \cite{Ka}).

\begin{defn}
Let $M$ be a model, let $B \subseteq M$ and $a \in M$. We say that
$\tp(a/B)$ \emph{splits} over a finite $A \subseteq B$ if there are finite tuples $c$ and $d$ in $B$ with
$\tp(c/A) = \tp(d/A)$ but
$\tp(c/A \cup a) \neq \tp( d/A \cup a)$.

We say that $\tp(a/B)$ \emph{Lascar splits} over $A$, if there
are $\vec b, \vec c\in B$ such that $\Lstp(\vec b/A)=\Lstp(\vec c/A)$ but $\tp(a\vec b/A)\neq \tp(a\vec c/A)$.

Similarly $\tp_H(a/B)$ \emph{splits} (resp Lascar splits) over a finite $A \subseteq B$ if there are finite tuples $c$ and $d$ in $B$ with
$\tp_H(c/A) = \tp_H(d/A)$  (resp. $\Lstp_H(c/A) = \Lstp_H(d/A)$) but
$\tp_H(c/A \cup a) \neq \tp_H( d/A \cup a)$.
\end{defn}

Clearly, if $\vec b,\vec c$ are tuples such that $\Lstp(\vec b/A)=\Lstp(\vec c/A)$ we also have $tp(\vec b/A)=tp(\vec c/A)$. If $A\in \mathcal{C}$, that is, if $A$ is a model, the converse holds. In general, types and Lascar strong types are different and the notions of splitting and Lascar splitting also disagree; let us consider an example.

\begin{example} Consider the setting from Example \ref{equiv-classes}. Fix $a\in M$ and let $b,c\in M$ with $aEc$, $aEc$ but all elements in the set $\{a,b,c\}$ are different. Then $\tp(b/a)=\tp(c/a)$, but $\Lstp(b/a)\neq \Lstp(c/a)$ as any model $N$ that contains the set $\{a\}$ also contains both $b,c$
and $\tp(b/N)\neq \tp(c/N)$. 

Now let $A=\{a'\in M: a'Ea\}$. Then $\tp(b/A)$ splits over $a$ since $\tp(b/a)=\tp(c/a)$ but $\tp(b,b/a)\neq \tp(b,c/a)$. But $\Lstp(b/A)$ does not split over $a$, whenever $\vec d_1,\vec d_2\in A$ are such that $\Lstp(\vec d_1/a)=\Lstp(\vec d_2/a)$ then $\vec d_1=\vec d_2$ and thus $\tp(b,\vec d_1/a)=\tp(b,\vec d_2/a)$.   
\end{example}

We will use the following notion from \cite[Definition 29]{Ka}.

\begin{defn}
For a tuple $\vec a\in M^n$ and a model $N\in \mathcal{C}$, we write $U_{Lsp}(\tp(\vec a/N))\geq n+1$
(respectively $U_{Lsp}(\tp_H(\vec a/N))\geq n+1$, 
where $(N,H)$ is an $H$-structure)
if there is $N_1\succeq N$ in $\mathcal{C}$ (resp. $(N_1,H)\succeq (N,H)$) such that $\tp(\vec a/N_1)$ Lascar splits over $N$ (resp. $\tp_H(\vec a/N)$ Lascar splits over $N_1$) and 
$U_{Lsp}(\tp(\vec a/N_1))\geq n$ (resp. $U_{Lsp}(\tp_H(\vec a/N_1))\geq n$). We write $U_{Lsp}(\tp(\vec a/N))\geq \omega$ (resp. $U_{Lsp}(\tp_H(\vec a/N))\geq \omega$) when $U_{Lsp}(\tp(\vec a/N))\geq n$ for all $n$
(resp. $U_{Lsp}(\tp_H(\vec a/N))\geq n$ for all $n$).

For an H-structure $(M,H)$, we write $U_{Lsp}((M,H))=1$ if for all $a\in M$ (singleton) and all submodels $(N,H)\preceq (M,H)$ we have $U_{Lsp}(\tp_H(a/N))\leq 1$. We write $U_{Lsp}((M,H))\geq\omega$ if there is $a\in M$ (singleton) and a submodel $(N,H)\preceq (M,H)$ such that $U_{Lsp}(\tp_H(a/N))\geq \omega$.

The same definitions can be applied to the class of beautiful pairs of structures in $\mathcal{C}$ instead of $H$-structures.
\end{defn}

Note that the definition above we could have used the word splitting instead of Lascar splitting as the parameters we are working with are \emph{models}. The definition we presented for the $U_{Lsp}$-ranks can be extended to work over sets as we now explain. Following the ideas from \cite{Ka}, we can define, for $A\subseteq M$ finite and for 
$\vec a\in M$, $U_{Lsp}(\tp(\vec a/A))=\max \{U_{Lsp}(\tp(\vec a/N)):A\subseteq N, N\in \mathcal{C}\}$ and for $A\subseteq M$ arbitrary, we define $U_{Lsp}(\tp(\vec a/A))=\min\{U_{Lsp}(\tp(\vec a/B)):B\subseteq A$ finite$\}$. In this paper we will deal mostly with $U_{Lsp}$-rank of tuples over models.

\begin{obse}
    It is proved in \cite[section 4]{Ka} that in a quasiminimal class the notion of independence coming from the absence of Lascar splitting agrees with the notion of independence coming from the pregeometry $\cl$. Moreover, working inside $\mathcal{C}$ a quasiminimal class, for $\vec a\in M^n$ and a set $B$, 
$U_{Lsp}(\tp(\vec a/B))=\dim_{\cl}(\vec a/B)$. This is why we deal in this paper mostly with Lascar splitting.
\end{obse}

In section \ref{sec:ranksgeometry} we will characterize the possible values for $U_{Lsp}$ in dense-codense expansions and tie the values of the rank with the complexity of the underlying geometry. We will now deal with splitting in expansions. We need the following fact from \cite{BHHKK}:

\begin{fact}\label{fin-oldnonsplit}(Proposition 4.2 \cite{BHHKK})
Let $M\in \mathcal{C}$ and let $B\subseteq M$ be a be countable closed submodel. For each finite tuple $\vec a \in M$ there is a finite $A \subseteq M$ such that $tp(\vec a/B)$ does not split over $A$.
\end{fact}

Now we generalize the previous result to the expansion.

\begin{prop}\label{superstability}
    
Let $M\in \mathcal{C}$ and assume that $(M,H)$ is an $H$-structure. Let $(B,H)\subseteq (M,H)$ be a countable ${\mathcal{L}_{H}}_{\omega_1, \omega}$-substructure. Then for each finite tuple $\vec a \in M$ there is a finite $A \subseteq B$ such that $tp_H(\vec a/B)$ does not split over $A$.
\end{prop}

\begin{proof}
By enlarging the tuple $\vec a$, may assume that $\vec a  \ind_{H(\vec a)} H(M)$. Since $(B,H)$ is a countable ${\mathcal{L}_{H}}_{\omega_1, \omega}$-substructure, by Proposition \ref{H_independentsubstructures} the set $B$ is $H$-independent in $(M,H)$.

First by Fact \ref{fin-oldnonsplit} we can find $A_1 \subseteq B$ finite such that $\tp(\vec a/B)$ does not split over $A_1$. Now let $A=A_1\cup HB(A_1/\vec a)$. Since $\vec a$ is $H$-independent we have $\vec a\ind_{H(\vec a)} H(M)$ and since $A=A_1\cup HB(A_1/\vec a)$, we have 
$A\ind_{\vec a H(A)} H(M)$, so $\vec aA\ind_{H(\vec a A)} H(M)$, this shows that the tuple $\vec aA$ is also $H$-independent.

\textbf{Claim}
$\tp_H(\vec a/B)$ does not split over $A$.

Choose $\vec c$ and $\vec d$ in $B$ with
$\tp_H(\vec c/A) = \tp_H(\vec d/A)$. By enlarging the tuples,
we may assume again that $\vec c$ and $\vec d$ are both $H$-independent over $A\vec a$. Then we have $\tp(\vec c/A) = \tp(\vec d/A)$ and since $\tp(\vec a/B)$ does not split over $A$, we also get $\tp(\vec c/A\vec a) = \tp(\vec d/A\vec a)$.

We then have $\tp(\vec c,A\vec a) = \tp(\vec d,A\vec a)$ (where we enumerate $A$ in the same way on both sides of the type), and both tuples are $H$-independent. Now apply Proposition \ref{H-typesomega1omega} to get that 
$\tp_H(\vec c,A\vec a) = \tp_H(\vec d,A\vec a)$.
\end{proof}

Let us look again at examples and let us consider categoricity in the extension.

\begin{example} Consider Example \ref{equiv-classes}. Let $(M,H)$ be an $H$-structure and let $B\subset M$ be at most countable and closed. Then $B$ is the union of at most countably many classes. Given $c\not \in B$ there are 
two options for $\tp(c/B)$, either the element $c$ belongs to a class that intersects $H(M)$ or $c$ belongs to a class that does not intersect $H(M)$. The class associated to the expansion is $2$-dimensional, parametrized by the cardinality of $H(M)$ and the cardinality of the classes that do not intersect $H(M)$. In particular the class of $H$-structures is not categorical on cardinals $\geq \aleph_1$.
\end{example}

The examples under consideration are $\omega$-stable, superstable in the sense of Proposition \ref{superstability} but not necessarily categorical in cardinals above $\aleph_0$ (see the previous example). The expansion fits into a generalized class described in terms of \emph{invariants of depth $0$} in the sense of Shelah (\cite{Sh}) but in the framework of $\mathcal{L}_{\omega_1\omega}(Q)$.

\section{Independence in H-structures and modular beautiful pairs}\label{sec:ranks}
Let $\mathcal{C}$ be a quasiminimal geometry class and write $\cl$ for the closure operator associated to the pregeometry. 
Given $M\in \mathcal{C}$, $\vec a\in M^n$ and $B\subseteq C\subseteq M$ we will now write $\vec a\ind^{\cl}_BC$ if $\dim_{\cl}(\vec a/B)=\dim_{\cl}(\vec a/C)$, where the dimension is calculated in the sense of the pregeometry $(M,\cl)$. When this is the case we say that $\vec a$ is \emph{$\cl$-independent from $C$ over $B$}. In \cite{Ka} Kangas studies independence in a quasiminimal geometry class using splitting and Lascar splitting. She shows that the notion of independence associated to Lascar splitting satisfies all the usual properties of an independence relation in a $\omega$-stable theory. In addition, she proves that if $\tp(\vec a/C)$ Lascar splits over $B\subseteq C$, then $\dim_{\cl}(\vec a/B)>\dim_{\cl}(\vec a/C)$ and vice versa, so the notion of independence associated to Lascar splitting agrees with the notion of independence associated with $\cl$-independence.


Our goal in this section is to study notions of independence that extends $\cl$-independence to the class of $H$-structures. We also study independence relations in beautiful pairs when the class $\mathcal{C}$ is modular.


\subsection{Independence in H-structures}
We now mix the ideas from \cite{AHKK} with the ones from  \cite[Section 5.2]{BeVa}.

\begin{defn} Fix $(M,H)$ an $H$-structure associated to the quasiminimal class $\mathcal{C}$. For sets $B,C\subseteq M$ and a tuple $\vec a\in M$ we write $\vec a \ind^{\cl,H}_{C}B$ if $$\vec a\ind^{\cl}_{C\cup H(M)}B \ \ \text{and} \ \  HB(\vec a/\cl(C))=HB(\vec a/\cl(C\cup B))$$

For sets $A,B,C\subseteq K$ we write $A \ind^{\cl,H}_{C}B$ if for all finite tuples 
$\vec a\in A$ we have $\vec a \ind^{\cl,H}_{C}B$.
   
\end{defn}

Note that the new notion of independence has two components, on the one hand the localization of the old notion of independence in $H(M)$ that explains the behaviour of elements independent from $H(M)$ and the $H$-basis component, that controls algebraicity over elements in $H(M)$.

Let us first prove that $\cl,H$-independence implies
$\cl$-independence:

\begin{prop}\label{newindep-oldindep}
Let $\vec a\in M^n$ and let $B\subseteq C \subseteq M$ and assume that $\vec a \ind^{\cl,H}_BC$. Then $\dim_{\cl}(\vec a/B)=\dim_{\cl}(\vec a/C)$. 
\end{prop}

\begin{proof}
We may write $\vec a=\vec a_1 \vec a_2$, where $\vec a_1$ is independent from $BH(M)$ and
$\vec a_2\in \cl(B\vec a_1H(M))$.

We prove that if $\dim_{\cl}(\vec a/B)>\dim_{\cl}(\vec a/C)$ then $\vec a \nind^{\cl,H}_BC$. By additivity of dimension $\dim_{\cl}(\vec a/B)=\dim_{\cl}(\vec a_1/B)+\dim_{\cl}(\vec a_2/B\vec a_1)$. It is enough to consider two cases.

Case 1. $\dim_{\cl}(\vec a_1/BH(M))>\dim_{\cl}(\vec a_1/CH(M))$. Then $\vec a_1 \nind^{cl,H}_BC$ as desired.

Case 2. $\dim_{\cl}(\vec a_1/BH(M))=\dim_{\cl}(\vec a_1/CH(M))$ and $\dim_{\cl}(\vec a/B)>\dim_{\cl}(\vec a/C)$. Since $\dim_{\cl}(\vec a_1/B)=\dim_{\cl}(\vec a_1/BH(M))=\dim_{\cl}(\vec a_1/CH(M))\leq \dim_{\cl}(\vec a_1/C)$, we must have $\dim_{\cl}(\vec a_1/B)=\dim_{\cl}(\vec a_1/C)$ and thus $\dim_{\cl}(\vec a_2/B\vec a_1)>\dim_{\cl}(\vec a_2/C\vec a_1)$. Then by additivity of H-basis (Proposition \ref{H-basis-additivity}) we get $HB(\vec a/C)=HB(\vec a_2/C\vec a_1)$ and $HB(\vec a/B)=HB(\vec a_2/B\vec a_1)$.

After reordering the tuple $\vec a_2$, we may write $\vec a_2=\vec a_{21}\vec a_{22}$ with $\dim_{\cl}(\vec a_{21}/B\vec a_1)=\dim_{\cl}(\vec a_{21}/C\vec a_1)$
and $\dim_{\cl}(\vec a_{22}/B\vec a_1\vec a_{21})>\dim_{\cl}(\vec a_{22}/C\vec a_1\vec a_{21})=0$.

Let $H_1=HB(\vec a_{21}/B\vec a_1)$, so $\vec a_{21}\in \cl(B\vec a_1H_1)$ and thus $\vec a_{21}\in \cl(C\vec a_1H_1)$, so also $\vec a_{2}\in \cl(C\vec a_1H_1)$ and thus
$HB(\vec a/C)\subseteq H_1$. Let $H_2=HB(\vec a_{22}/B\vec a_1\vec a_{21})$, since $\dim_{\cl}(\vec a_{22}/B\vec a_1\vec a_{21})>0$ we get $H_2\neq \emptyset$.
Using again additivity of $H$-basis (Proposition \ref{H-basis-additivity}) $HB(\vec a/C)=H_1 \cupdot H_2$. We conclude $HB(\vec a/C)\subsetneq HB(\vec a/B)$ and we get $\vec a \nind^{\cl,H}_{B}C$ as desired. 
\end{proof}

For tuples in $H$ we get a very simple description of $\cl,H$-independence.

\begin{cor}
Let $\vec a\in H(M)^n$ and let $B\subseteq C \subseteq M$. Then $\vec a \ind^{\cl,H}_BC$ iff $\dim_{\cl}(\vec a/B)=\dim_{\cl}(\vec a/C)$. 
\end{cor}

\begin{proof}
By the previous result, if $\dim_{\cl}(\vec a/B)>\dim_{\cl}(\vec a/C)$ then $\vec a \nind^{\cl,H}_BC$. Now assume that $\dim_{\cl}(\vec a/B)=\dim_{\cl}(\vec a/C)$. Note that if $\vec a\in H(M)^n$, we always have $\vec a\ind^{\cl}_{B\cup H(M)}CH(M)$. Assume $B=\cl(B)$ and that $B$ is $H$-independent. Then we can write $\vec a=\vec a_1\vec a_2$ with $\vec a_1$ an independent tuple from $B$ and $\vec a_2\in \cl(B\vec a_1)$. Then $HB(\vec a/B)=\vec a_1$ (modulo permutations). If $\dim_{\cl}(\vec a/B)=\dim_{\cl}(\vec a/C)$, then $HB(\vec a/C)=\vec a_1$ as desired.
\end{proof}

\begin{theo}
The notion of independence $\ind^{\cl,H}$ satisfies invariance, normality, monotonicity, transitivity, symmetry, finite character and extension.   
\end{theo}

\begin{proof}
Invariance follows from the fact that  $\ind^{\cl}$ satisfies invariance and the fact that the notion of $H$-basis is also invariant under automorphisms. 

Normality, existence, monotonicity and transitivity are clear. Finite character follows from definition.

Symmetry. Since $\ind^{\cl}$ satisfies symmetry, so does its localization on $H(M)$. Now we check that for finite tuples $\vec a,\vec b\in M$, whenever $HB(\vec a/\cl(C))=HB(\vec a/\cl( C\cup \vec b))$ then 
$HB(\vec b/\cl(C))=HB(\vec b/\cl( C\cup \vec a))$.
Without loss of generality, $C$ is closed and $H$-independent. Now assume that $HB(\vec a/\cl(C))=HB(\vec a/\cl( C\cup \vec b))$. By Proposition \ref{H-basis-additivity} and this last assumption we have $$HB(\vec a \vec b/C)=HB(\vec b /C)\cupdot  HB(\vec a/\cl(C\vec b))=HB(\vec b /C)\cupdot HB(\vec a/C)$$

Again by Proposition \ref{H-basis-additivity} we have $HB(\vec a \vec b/C)=HB(\vec a /C)\cupdot HB(\vec b/\cl(C\vec a))$. Putting the two equalities together we get $HB(\vec b/\cl(C\vec a))=HB(\vec b/C)$ as desired.

\end{proof}

We are missing two key properties: existence and stationarity. To prove the second property, we will use Fact \ref{stationarityinC}, namely, the notion of $\cl$-independence satisfies stationarity over models in $\mathcal{C}$.


We can now prove our main theorems of this subsection:

\begin{theo}\label{stationarityinH}
The notion of independence $\ind^{\cl,H}$ satisfies stationarity over models $M\in \mathcal{C}$ satisfying $M\subset N$, $(N,H)$ an $H$-structure and $M$ is $H$-independent inside $N$.   
\end{theo}

\begin{proof}
Let $(N,H)$ be an $H$-structure associated
to the class $\mathcal{C}$. Let $M\subseteq C\subseteq N$ and assume $M\in \mathcal{C}$ and that $M$ is $H$-independent. Let $\vec a=\vec a_1 \vec a_2 \vec a_3$, $\vec c=\vec c_1 \vec c_2 \vec c_3$ be tuples such that 
$\tp_H(\vec a/M)=tp_H(\vec c/M)$, $\vec a_1$ is independent over $MH(N)$, $\vec a_2$ is independent over $M\vec a_1$ with $\vec a_2\in \cl(\vec a_1MH(N))$ and $\vec a_3\in \cl(M\vec a_1\vec a_2)$ and assume there is a similar decomposition for $\vec c$. Finally assume that $\vec a \ind^{\cl,H}_MC$, $\vec c \ind^{\cl,H}_M C$. We may assume, after enlarging the tuples, that $\vec a_2,\vec c_2\in H(N)$. We may also assume, after enlarging $C$, that $C$ is $H$-independent over $M$. 

By Proposition \ref{newindep-oldindep} and stationarity of $\ind^{\cl}$ over models in $\mathcal{C}$ (Fact \ref{stationarityinC}) we get 
 $\tp(\vec a_1 \vec a_2 \vec a_3/C)=\tp(\vec c_1\vec c_2 \vec c_3/C)$.
Recall that we enlarged $C$ so that $C=M\cup C$ is $H$-independent.

 \textbf{Claim} The tuple $\vec a$ is $H$-independent over $C$.

We have $\vec a_1\ind^{cl} H(N)C$ and $\vec a_2\in H(N)$, so $\vec a_1\vec a_2 \ind^{cl}_{C\vec a_2} H(N)$ and thus we also get $\vec a_1\vec a_2\vec a_3 \ind^{cl}_{C\vec a_2} H(N)$, from which it follows that $\vec a \ind^{cl}_{CH(\vec a)} H(N)$.

Similarly, we have that $\vec c$ is $H$-independent over $C$. 
 
We have shown that the set $\vec aC$ is $H$-independent, $\vec cC$ is $H$-independent and $\tp(\vec a H(\vec a)/C)=tp(\vec c H(\vec c)/C)$. So by Proposition \ref{H-typesomega1omega} we get $\tp_H(\vec a/C)=tp_H(\vec c/C)$ as desired.

\end{proof}

\begin{theo}
The notion of independence $\ind^{\cl,H}$ satisfies extension over sets.   
\end{theo}

\begin{proof}
We fix $\kappa$ an infinite cardinal and work in $(M,H)$ an $H$-structure with $\dim(H(M))=\kappa$ and  $\dim(M/H(M))=\kappa$.
 Let $\vec a\in M$ and let $B\subseteq C$ be sets with $|B|,|C|<\kappa$. We will show $\tp_H(\vec a/B)$ has a free extension (in the sense of $\ind^{\cl,H}$) to $C$. By the definition of $\ind^{\cl,H}$-independence, we may assume that $B=B\cup HB(B)$. Write $\vec a=\vec a_1 \vec a_2 \vec a_3$, where $\vec a_1$ is independent over $BH(M)$, $\vec a_2$ is independent over $B\vec a_1$ with
$\vec a_2\in \cl(B\vec a_1H(M))$ 
and $\vec a_3\in \cl(\vec a_1\vec a_2B)$. By adding $HB(\vec a/B)$ to the tuple $\vec a$ and replacing the tuple for a larger one if necessary, we may assume that $\vec a_2\in H(M)$ and that it is an independent tuple of elements. Since $\dim(M/H(M))=\kappa$ and $|C|<\kappa$, we can find a tuple $\vec c_1$ which is independent over $CH(M)$ and of the same length as $\vec a_1$.  Since $\dim(H(M))=\kappa$ and $|C|<\kappa$ we can find a tuple $\vec c_2\in H(M)$ which is independent over $C$ and has the same length as $\vec a_2$. Notice that $B$ is $H$-independent and both $\vec a_1\vec a_2$ and $\vec c_1\vec c_2$ are $H$-independent over $B$, so $\tp_H(\vec a_1\vec a_2/B)=\tp_H(\vec c_1\vec c_2/B)$. 

By Proposition \ref{superstability} we can find $B_0\subseteq B$ be finite such that $\tp_H(\vec a_1\vec a_2 \vec a_3/B)$ does not Lascar split over $B_0$. Clearly $\tp(\vec a_1\vec a_2 /B_0)=\tp(\vec c_1\vec c_2 /B_0)$ so by homogeneity in $\mathcal{C}$ over $\emptyset$ (axiom QM5) there is $\vec c_3$ such that $\tp(\vec a_1\vec a_2 \vec a_3/B_0)=\tp(\vec c_1\vec c_2 \vec c_3/B_0)$.

\textbf{Claim} $\tp_H(\vec a_1\vec a_2 \vec a_3/B)=\tp_H(\vec c_1\vec c_2 \vec c_3/B)$.

Let $\vec b\in B$ be finite. Then $\vec a_1\vec a_2 \vec a_3\ind^{\cl}_{B_0}\vec b$, $\vec c_1\vec c_2 \vec c_3\ind^{\cl}_{B_0}\vec b$ and thus $\tp(\vec a_1\vec a_2 \vec a_3/B_0\vec b)=\tp(\vec c_1\vec c_2 \vec c_3/B_0\vec b)$.

Let $f$ be  a partial $\mathcal{L}_{\omega_1 \omega}$-map fixing $B$ pointwise and sending $\vec a_1\vec a_2\vec a_3$ to $\vec c_1\vec c_2 \vec c_3$. We show that for any $e\in \vec a_3$, $e\in H(M)$ iff $f(e)\in H(M)$. Indeed, if $e\in H(M)$, since $\vec a$ is $H$-independent over $B$, either $e\in \vec a_2$ or $e\in H(B)$.  If $e\in H(B)$, since $f$ fixes $B$ pointwise we get $f(e)=e\in H(M)$. If 
$e\in \vec a_2$, since $\tp_H(\vec a_1\vec a_2/B)=\tp_H(\vec c_1\vec c_2/B)$ we have $f(e)\in \vec c_2$ and $f(e)\in H$.

Since $\tp(\vec a_1\vec a_2 \vec a_3/B)=\tp(\vec c_1\vec c_2 \vec c_3/B)$, $B$ is $H$-independent, $\vec a_1\vec a_2 \vec a_3$ is $H$-independent over $B$ and $\vec c_1\vec c_2 \vec c_3$ is $H$-independent over $B$ and 
$\tp(\vec aH(\vec a)/B)=\tp(\vec cH(\vec c)/B)$ by Proposition \ref{H-typesomega1omega} we get $\tp_H(\vec a/B)=\tp_H(\vec c/B)$. Then $\vec c=\vec c_1 \vec c_2 \vec c_3$ realizes the desired extension.


\end{proof}

Since the notion of independence $\ind^{\cl,H}$ has all the desirable properties of stable independence notion, it should coincide with other robust notions of independence. Just as Kangas \cite[section 4]{Ka} showed that independence in the sense of Lascar splitting in $\mathcal{C}$ agrees with $\cl$-independence, here we show that $\ind^{\cl,H}$ agrees with the notion of independence coming from splitting over models (which coincides with Lascar splitting).

\begin{theo}\label{indepegreessplittingH}
    Let $(M_1,H)\preceq (M_2,H) \preceq (M,H)$ be $H$-structures and let $\vec a\in M$.
   Assume $\vec a\nind^{\cl, H}_{M_1} M_2$.
   \begin{enumerate}

\item Then there is a finite $B\subseteq M_1$ such that $\tp_H(\vec a/M_1)$ does not split over $B$ and $\tp_H(\vec a/M_2)$ splits over $B$.
   
\item Assume furthermore that $\dim(H(M_2)/M_1)=\infty$ and $\dim(M_2/M_1H(M_2))=\infty$. Then $\tp_H(\vec a/M_2)$ splits over $M_1$

\end{enumerate}

\end{theo}

\begin{proof} We prove (1). Assume $\vec a\nind^{\cl, H}_{M_1} M_2$. Let $B\subseteq M_1$ be finite such that $\tp_H(\vec a/M_1)$ does not split over $B$. By enlarging $B$ if necessary, we may assume $B=B\cup HB(B)$, so $B$ is $H$-independent. We may also assume that $\dim(\vec a/M_1H(M_2))=\dim(\vec a/BH(M_2))$. Finally note that since $B$ is finite we have that $\dim(M_2/H(M_2)B)=\infty$ and that $\dim(H(M_2)/B)=\infty$.

\textbf{Case 1} Assume first that $\dim(\vec a/M_2H(M_2))<\dim(\vec a/M_1H(M_2))=\dim(\vec a/BH(M_2))$. Let $\vec c\in M_2$ be a finite tuple $\cl$-independent over $M_1H(M_2)$ such that $\dim(\vec a/\vec cM_1H(M_2))=\dim(\vec a/M_2)$ and let $\vec d\in M_2$ be another tuple of the same length as $\vec c$ which is $\cl$-independent over $BH(M_2)\vec a$. Then $\dim(\vec a/BH(M_2))=\dim(\vec a/\vec d BH(M))$. We have $\tp_H(\vec c/B)=\tp_H(\vec d/B)$, but $\tp_H(\vec a,\vec c/B)\neq \tp_H(\vec a, \vec d/B)$ as the first type says there is a tuple in $H(M)$ that witnesses $\dim(\vec a/\vec cBH(M))<\dim(\vec a/BH(M))$, which is not the case for $\tp_H(\vec a, \vec d/B)$.

\textbf{Case 2} Assume now that $\dim(\vec a/M_2H(M))=\dim(\vec a/M_1H(M))$ but 
we have $HB(\vec a/M_2)\neq HB(\vec a/M_1)$. By minimality of the $H$-basis we must have $HB(\vec a/M_2)\subsetneq HB(\vec a/M_1)$. Let $\vec c\in M_2$ be a finite tuple $H$-independent over $M_1$ such that $HB(\vec a/M_2)=HB(\vec a/\cl(M_1\vec c))$ and let $\vec d\in M_2$ be another tuple with $\tp_H(\vec d/M_1)=\tp_H(\vec c/M_1)$ such that $\vec d\ind^{\cl, H}_{M_1} \vec a M_1$. Then $$HB(\vec a/M_2)=HB(\vec a/M_1\vec c)\subsetneq HB(\vec a/M_1)=HB(\vec a/M_1\vec d)$$ and thus $\tp_H(\vec a,\vec c/M_1)\neq \tp_H(\vec a, \vec d/M_1)$.

For (2), note that the only ingredient that we used in the proof was that $B$ is $H$-independent, $\dim(M_2/H(M_2)B)=\infty$ and that $\dim(H(M_2)/B)=\infty$. These assumptions hold exchanging $B$ for $M_1$ under the extra hypothesis.
\end{proof}

We will use part (2) of the previous Theorem, part (1) is analogous to Definition 42 in \cite{Ka}. We can prove a converse of the previous result.

\begin{theo}\label{indepegreessplittingHp2}
    Let $(M_1,H)\preceq (M_2,H) \preceq (M,H)$ be $H$-structures and let $\vec a\in M$.
   Assume that $\vec a\ind^{\cl, H}_{M_1} M_2$,
Then $\tp_H(\vec a/M_2)$ does not split over $M_1$.
\end{theo}

\begin{proof}
Assume that $\vec a\ind^{\cl, H}_{M_1} M_2$,
then by symmetry $M_2\ind^{\cl, H}_{M_1} \vec a$. Let $\vec c,\vec d \in M_2$
 with $\tp_H(\vec c/M_1)=\tp_H(\vec d/M_1)$. As $(M_1,H)$ is a model, by stationarity of $\ind^{\cl, H}$ we get $\tp_H(\vec c/M_1\vec a)=\tp_H(\vec d/M_1\vec a)$.
\end{proof}

We now provide some examples where we compare $\ind^{\cl}$
and $\ind^{\cl, H}$ for different $H$-structures.

\begin{example}\label{Torsionfreegroups}
 Let $\mathcal{L}_{groups}$ be the language of additive groups $\{+,-,0\}$. Consider the class $\mathcal{C}$ whose objects consist of torsion free divisible abelian groups. The operator $\cl$ is interpreted as group-theoretic algebraic closure which coincides with the group-theoretic definable closure. Since the theory of vector spaces over $\mathbb{Q}$ is strongly minimal, it is a geometric first order theory. Let $\mathcal{C}$ be the class of those vector spaces that have infinite dimension, note that this last requirement guarantees axiom $(QM2)$. For each $M\in \mathcal{C}$, we build an $H$-structure by interpreting $H(M)$ as a infinite collection of independent vectors such that $dim(M/H(M))=\infty$. 

 Let $a\in M$ be such that $a\not  \in \cl(H(M))$ and let $h\in H(M)$. Then $a,h$ are $\cl$-independent and let $b=a+h$. Note that $a, b$ are again $\cl$-independent and thus  $a\ind^{\cl} b$. But $a\in \cl(bH(M))\setminus  \cl(H(M))$, so  $a\nind^{\cl,H} b$.
\end{example}

\begin{example}\label{equiv-classesrevisited}(Example \ref{equiv-classes} revisited)
 Let $\mathcal{L}=\{E\}$ and consider the class $\mathcal{C}$ whose objects consist of infinite sets and $E$ is interpreted as an equivalence relation with infinitely many classes, all of whose blocks have size $\aleph_0$ and the closure of a point is the class where it belongs. This example has a \emph{trivial} geometry. Assume that $\vec a$, $\vec b$ are $\cl$-independent, i.e.  $\vec a\ind^{\cl} \vec b$.
Note that $\dim_{\cl}(\vec a/H(M))$ is the number of classes intersecting $\vec a$ that do not intersect $H(M)$. Since $\vec a$, $\vec b$ are $\cl$-independent, the classes intersecting $\vec a$ are disjoint from the classes intersecting $\vec b$ and thus $\dim_{\cl}(\vec a/H(M))=\dim_{\cl}(\vec a/\vec b H(M))$. It also follows that $HB(\vec a)=HB(\vec a/\vec b)$ and thus
 $\vec a\ind^{\cl,H} \vec b$.
\end{example}

The picture behind the previous two examples resembles what happens in strongly minimal theories, when $\cl$ is trivial, $\ind^{\cl,H}$ coincides with $\ind^{\cl}$, but
when $\cl$ is non trivial they differ. This fact will be proved in Corollary  \ref{Lascarequals} in the next section.

\begin{example}(Zilber's exponential fields)
Let us explain what the results of this section mean in the main example we have in mind, namely the class associated to Zilber's complex exponential field \cite{Zi}. Let us recall some definitions from \cite{AHKK}. We will assume the reader is familiar with the notion of exponential fields as explained in \cite{AHKK}. Fix $K$ an exponential field.

\begin{defn}
A Khovanskii system (of equations and
inequations) consists of, for some $n \in \mathbb{N}$,  exponential polynomials $f_1, \dots , f_n \in
K[x_1,\dots, x_n]$, with equations $f_i(x_1,\dots, x_n) = 0$ for $i = 1, \dots , n$ and the inequation $Jac(x_1,\dots, x_n) \neq 0$ where $Jac$ is the Jacobian of the functions $f_1, \dots , f_n$ with respect to the variables $x_1,\dots, x_n$. By exponential polynomials, we mean functions of the form  $f_1, \dots , f_n \in
K[x_1,\dots, x_n,e^{x_1},\dots,e^{x_n}]$, so they are polynomials in the variables $x_1,\dots, x_n,e^{x_1},\dots,e^{x_n}$.
\end{defn} 

\begin{defn}

For a set $C\subseteq K$ and $a_0\in K$, we write $a_0\in \ecl(C)$ if there are $f_1,\dots,f_n\in \mathbb{Q}(C)[x_1,\dots,x_n,e^{x_1},\dots,e^{x_n}]$ and $a_1,\dots,a_{n-1}\in K$
such that $a_0,a_1,\dots,a_{n-1}$ is a solution of the Khovanskii system associated to $f_1,\dots,f_n$. That is, 
$f_i(a_0,a_1,\dots,a_{n-1})=0$ for $i\leq n$ and $Jac(a_0,\dots, a_{n-1}) \neq 0$ 
\end{defn} 

For exponential fields, including Zilber's field, we have:

\begin{fact}\cite{KiEx}
The closure operator $\ecl$ defines a pregeometry in all exponential fields.
\end{fact}

When $\mathcal{C}$ is the class of Zilber fields, $\cl=ecl$. Thus $\cl$-independence agrees with the notion of independence studied in \cite[section 5]{AHKK} called $etd$-independence. By Corollary \ref{stationarityinC} in Zilber's field,
$\ind^{etd}$ satisfies \emph{stationarity} over models. This is not true in all exponential fields, for example, in $\mathbb{R}_{exp}$ (the reals expanded by exponentiation) algebraic independence agrees with $etd$-independence but clearly types are not stationary.

We will now consider elementary pairs, where the smaller model is closed inside the big one. The model-theoretic study of pairs of algebraically closed fields goes back to the work of Robinson \cite{Ro} and Keisler \cite{Ke}. If $(K,K_0)$ is a pair of models of $ACF_0$ and $K_0\subsetneq K$, then $(K,K_0)$ is a model of the theory of beautiful pairs of $ACF_0$.
It seems interesting to see what kind of pairs we get in this stronger setting as the expressive power of the language is stronger. 

\begin{rema}
Let $\mathcal{C}$ be the class of Zilber's field and let $(K,K_0)$ be a pair of models of $\mathcal{C}$ with $K_0\preceq K$.
Then for every $2\leq n\leq \infty$ there is a unique $\mathcal{L}_{\omega_1 \omega}(Q)$-theory
whose models satisfy $dim(K/K_0)=n$. 
\end{rema}

\begin{proof}
    
 We will first show that we can axiomatize pairs that satisfy $\dim(K/K_0)=n$ for some $1<n< \infty$.
Consider the sentence $\varphi_n$ given by
$$\exists x_1\dots \exists x_n \bigwedge_{m\geq 1} \Big{(}\forall y_1\in P \dots \forall y_m\in P \bigwedge_{0\leq i\leq n}\neg \pi_{m+i}(x_{i+1};x_1,\dots,x_i,y_1,\dots,y_m)\Big{)}$$
The sentence $\varphi_n$ says that there are at least $n$ elements in $K$ which are independent over $K_0$. We then have $\dim(K/K_0)=n$ iff $(K,K_0)\models \varphi_n\wedge \neg \varphi_{n+1}$. A beautiful pair need to satisfy $(K,K_0)\models \bigwedge_{n \geq 1}\varphi_n$.

\end{proof}
\end{example}

\subsection{Independence in beautiful pairs, the locally modular case}

We now introduce a notion of independence in beautiful pairs of \emph{modular} quasiminimal classes, the locally modular classes can be handled in a similar way after localizing. We will use the results of this subsection in section \ref{sec:ranksgeometry} when we study the $U_{Lsp}$ in modular quasiminimal classes. For the rest of this section we will only consider quasiminimal modular classes.

\begin{defn} Fix $(M,P)$ a beautiful pair associated to the quasiminimal modular class $\mathcal{C}$. Let $B\subseteq M$ and $\vec a\in M^n$. We write $\vec a \in \scl(B)$ if 
$\vec a\in \cl(BP(M))$. We write $\dim_{\scl}(\vec a/B)$ for $\dim(\vec a/B)$ in structure $(M,\cl_P)$, where the closure operator is localized in $P(M)$.
\end{defn}

\begin{defn} Fix $(M,P)$ a beautiful pair associated to the quasiminimal modular class $\mathcal{C}$. For sets $B\subseteq C\subseteq M$ and a tuple $\vec a\in M$ we write $\vec a \ind^{\cl,P}_{B}C$ if 
\begin{enumerate}
\item $\dim(\vec a/B)=\dim(\vec a/C)$ (i.e. $\vec a \ind^{\cl}_{B}C$)
\item $\dim_{\scl}(\vec a/B)=\dim_{\scl}(\vec a/C)$ (i.e. $\vec a \ind^{\cl}_{BP(M)}C$)
\end{enumerate}
For sets $A,B,C\subseteq M$, we write $A \ind^{\cl,P}_{B}C$ if $\vec a \ind^{\cl,P}_{B}B\cup C$ for all finite tuples $\vec a\in A$. If this is the case, we say that  $A$ is $\cl,P$-independent from $B$ over $C$.
\end{defn}

\begin{rema}
The previous description of $\cl,P$-independence corresponds to forking-independence in the setting of lovely pairs of SU-rank one linear theories. It appears in the proof of Theorem 5.13
($(i \implies ii)$) in \cite{v2003} for $1$-types, but the argument generalizes to general types.
\end{rema}

\begin{obse}\label{addingpreservesindependence}
    
Let $(M,P)$ be a beautiful pair associated to the quasiminimal class $\mathcal{C}$ and assume that $\vec a \ind^{\cl,P}_{B}C$. 
Then whenever $\vec d\in \cl(B\vec a)$ and $C\subseteq D\subseteq \cl(B\cup C)$ we also have $\vec d \ind^{\cl,P}_{B}D$. 
\end{obse}

Since the definition of $\cl,P$ independence comes from equalities of dimensions, it satisfies invariance, finite character, local character and transitivity. 

\begin{lem}
    The notion of independence $\ind^{\cl,P}$ satisfies symmetry.
\end{lem}

\begin{proof}
    Assume $\vec a,\vec b$ are tuples and let $C\subseteq M$. If $\vec a \ind^{\cl,P}_{C} \vec b$,
    then $\dim(\vec a/C)=\dim(\vec a/C\vec b)$ 
and $\dim_{\scl}(\vec a/C)=\dim_{\scl}(\vec a/C \vec b)$. By additivity of dimension
    we have $\dim(\vec a \vec b/C)=\dim(\vec a/C\vec b)+\dim(\vec b/C)$ and
    $\dim(\vec a \vec b/C)=\dim(\vec b/C\vec a)+\dim(\vec a/C)$.
If $\dim(\vec a/C\vec b)=\dim(\vec a/C)$ then $\dim(\vec b/C\vec a)=\dim(\vec b/C)$. The same argument works for the dimension localized in $P(M)$.
This shows $\vec b \ind^{\cl,P}_{C} \vec a$.
\end{proof}

We will need the following technical observation.

\begin{lem}\label{bothindepPindep}
Let $\mathcal{C}$ be modular, $M\in \mathcal{C}$ and $(M,P)$ a lovely pair. Let $(N,P)\preceq (M,P)$ and let  $A,B\subseteq M$ be both $P$-independent over $N$. If $A\ind_N^{\cl, P}B$ then $A\cup B$ is also $P$-independent over $N$.
\end{lem}

\begin{proof}
We have $P(A\cup B)=P(A)\cup P(B)$. To show $A\cup B$ is $P$-independent, by transitivity it is enough to show $A\ind_{NP(A)P(B)}^{\cl}P(M)$ and $B\ind_{NP(B)A}^{\cl}P(M)$.
Since $A\ind_{NP(A)}^{\cl}P(M)$ we get $A\ind_{NP(A)P(B)}^{\cl}P(M)$.

On the other hand, $B\ind_N^{\cl, P}A$ implies 
$B\ind_{NP(M)}^{\cl}A$ and since $B\ind_{NP(B)}^{\cl}P(M)$ we get $B\ind_{NP(B)}^{\cl}AP(M)$, from which  $B\ind_{NAP(B)}^{\cl}P(M)$.
\end{proof}

We are ready to prove our main results. 


\begin{prop}
    The notion of independence $\ind^{\cl,P}$ satisfies stationarity.
\end{prop}

\begin{proof}
    Assume $\vec a,\vec b$ are tuples in a beautiful pair $(M,P)$ where $M\in \mathcal{C}$ such that $dim(M/P(M)\geq \kappa$ and $dim(P(M)\geq \kappa$ for $\kappa$ sufficiently large. Let $N\in \mathcal{C}$ be small and assume $(N,P)\preceq (M,P)$. Assume that $\tp_P(\vec a/N)=\tp_P(\vec b/N)$.
    Now let $C\subseteq M$ be small with $\vec a \ind^{\cl,P}_{N}C$ and $\vec b\ind^{\cl,P}_{N}C$.

 Since $\mathcal{C}$ is modular, $\vec aN \ind_{\cl(\vec aN)\cap P(M)} P(M)$, so we can find $\vec d\in \cl(\vec a N)\cap P(M)$ with $\vec a \ind_{N\vec d} P(M)$. Let $\vec a'=\vec a\vec d$, then $\vec a'$ is $P$-independent over $N$ and since $\vec d\in \cl(N\vec a)$ by Observation \ref{addingpreservesindependence} we still have $\vec a' \ind^{\cl,P}_{N}C$. Let $\vec e\in \cl(\vec bN)\cap P(M)$ be such that $\tp_P(\vec a \vec d/N)=\tp_P(\vec b \vec e/N)$. 
 Let $\vec b'=\vec b \vec e$, since $\vec b'\in \cl(\vec b)$ by Observation \ref{addingpreservesindependence} we also have $\vec b' \ind^{\cl,P}_{N}C$. 
Finally, by enlarging $C$ with a subset of $\cl(C)\cap P(M)$ if necessary, we may also assume that $C$ is $P$-independent over $N$ and note that by Observation \ref{addingpreservesindependence} we still have $\vec a' \ind^{\cl,P}_{N}C$, $\vec b' \ind^{\cl,P}_{N}C$.

Note that $\tp(\vec a'/N)=\tp(\vec b'/N)$, $\vec a' \ind^{\cl}_{N}C$ and $\vec b'\ind^{\cl}_{N}C$, so by stationarity in $\mathcal{C}$ we have $\tp(\vec a',C/N)=\tp(\vec b',C/N)$. 

Choosing the same enumeration for $N$ in both tuples
we get $\tp(\vec a',C,N)=\tp(\vec b',C,N)$. Note that $(N,P)\preceq (M,P)$. We then have that $\vec a'$ is $P$-independent over $N$, $C$ is $P$-independent over $N$ and  $\vec a' \ind^{\cl}_{N}C$ so by Lemma \ref{bothindepPindep} we get that $\vec a'CN$ is $P$-independent. Similarly, $\vec b'CN$ is $P$-independent. Since $\tp_P(\vec a')=\tp_P(\vec b')$ for any $a_i\in \vec a'$ and corresponding $b_i\in \vec b'$ we have $P(a_i)$ if  and only if $P(b_i)$. We conclude $\tp_P(\vec a',C,N)=\tp_P(\vec b',C,N)$ as desired. 
    
\end{proof}

\begin{prop}
    The notion of independence $\ind^{\cl,P}$ satisfies existence.
\end{prop}

\begin{proof}
Let $(M,P)$ be a beautiful pair where $dim(M/P(M)\geq \kappa$ and $dim(P(M)\geq \kappa$ for $\kappa$ sufficiently large. Let $C\subseteq D\subseteq M$ be small and let $\vec a\in M$ be a tuple. We will show that there is $\vec a'$ with  $\tp_P(\vec a'/C)=\tp_P(\vec a'/C$ with $a'\ind^{\cl,P}_CD$.
We may enlarge $C$ by adding elements in $\cl(C)$ if necessary so that $C$ is $P$-independent. Similarly, we may add elements to $\vec a$ so that it is $P$-independent over $C$. Write $\vec a=\vec a_1\vec a_2\vec a_3$ so that $\vec a_1$ is $\cl$-independent over $CP(M)$, $\vec a_2$ is $\cl$-independent over $C$ and $\vec a_2\in P(M)$ and $\vec a_3\in \cl(C\vec a_1\vec a_2)$.   Again, by adding elements to $D$ we can assume $D$ is $P$-independent. Using existence in $\mathcal{C}$ we can find $\vec a'=\vec a_1'\vec a_2'\vec a_3'$ with $\tp(\vec a'/C)=\tp(\vec a/C)$ and $\vec a'\ind_C D$. By the density property we may assume $\vec a_2'\in P(M)$ and by the codensity property we may assume $\vec a_1'\ind DP(M)$.
Then $\dim(\vec a'/D)=|\vec a_1'|+|\vec a_2'|=\dim(\vec a'/C)$ and $\dim_{\scl}(\vec a'/D)=|\vec a_1'|=\dim_{\scl}(\vec a'/C)$, so $\vec a'\ind^{\cl,P}_CD$.
\end{proof}

\begin{theo}\label{indepegreessplittingP}
    Let $(M_1,P)\preceq (M_2,P) \preceq (M,P)$ be beautiful pairs and let $\vec a\in M$. Assume furthermore that $\dim(P(M_2)/M_1)=\infty$ and $\dim(M_2/M_1P(M_2))=\infty$. Then $\vec a\nind^{\cl, P}_{M_1} M_2$ iff $\tp_P(\vec a/M_2)$ splits over $M_1$.  
\end{theo}

\begin{proof} $(\Rightarrow)$
The proof is similar to the argument presented in Theorem \ref{indepegreessplittingH}. We consider two cases.

Case 1.  $\dim(\vec a/M_2P(M))<\dim(\vec a/M_1P(M))$. We choose tuples of independent elements $\vec c\in M_2$ and $\vec d\in M_2$ of the same length both independent from $M_1P(M_2)$ with $\dim(\vec a/M_2P(M))=\dim(\vec a/\vec cM_1P(M))$ and $\vec d\ind^{cl,P}P(M_2)M_1\vec a$. 
Then $\tp_P(\vec c/M_1)=\tp_P(\vec d/M_1)$ but 
$\tp_P(\vec c\vec a/M_1)\neq \tp_P(\vec d\vec a/M_1)$.

Case 2. $\dim(\vec a/M_2)<\dim(\vec a/M_1)$. Let $a_i\in \vec a$ be 
such that $a_i\in \cl(M_2 \vec a\setminus a_i)$ and $a_i\not \in \cl(M_1 \vec a\setminus a_i)$. Now let $d\models \tp_P(a_i/M_1)$ with $d\ind^{\cl,P}_{M_1} \vec a$. We get that $\tp(\vec a,a_i/M_1)\neq \tp(\vec a,d/M_1)$ and also $\tp_P(\vec a,a_i/M_1)\neq \tp_P(\vec a,d/M_1)$.

 $(\Leftarrow)$ This direction follows from stationarity and symmetry of $\ind^{\cl, P}$, it is a similar argument to the one presented in Theorem \ref{indepegreessplittingHp2}. For this direction we do not need the saturation assumption.
\end{proof}

It is an interesting problem how to characterize independence in beautiful pairs that are not modular.
The main obstacle in the non-locally modular case is that the characterization given in \cite{BPV} uses canonical bases, which, in general, do not live as real elements in a quasiminimal class. 

\begin{ques}
Can the approach from \cite{HK} to construct canonical bases be used to characterize Lascar
splitting analogously to \cite[Proposition 7.3]{BPV}?
\end{ques}

\section{Ranks and geometry in $H$-structures and beautiful pairs}\label{sec:ranksgeometry}

In this section, we follow ideas from \cite{BeVaLP, BeVa} and relate the $U_{Lsp}$-rank of the dense-codense expansions to the complexity of underlying of the geometry in the class $\mathcal{C}$. We also generalize ideas from \cite{v2003} and characterize 
locally modular quasiminimal classes in terms of properties of the associated class of beautiful pairs (see Theorem \ref{linearity} below).

We start with $H$-structures.
For this part of the paper we need a basic observation. If $(N,H)\preceq (M,H)$ are $H$-structures and $\vec a\in H(M)$,
then $\cl(N\vec a)\in \mathcal{C}$ and it can be seen as an $H$-structure by setting $H(\cl(N\vec a))=H(N)\cup\{\vec a\}$. Furthermore with this interpretation we have $(\cl(N\vec a),H)\preceq (M,H)$. We also need a definition.

\begin{defn}
Consider $(M,\cl)$ as a pregeometry and let $A=\{a_1,\dots,a_n\}$. We say $A$ is a \emph{$\cl$-$n$-gon} or just a \emph{circuit} if it is a minimally dependent set, so for any $1\leq i \leq n$, $A\setminus \{a_i\}$ is independent but 
$a_i\in \cl(A\setminus \{a_i\})$.
\end{defn}

We reproduce the argument from \cite[section 3]{BeVa}.

\begin{prop}\label{existence ngon}
Let $(M,\cl)$ be a non-disintegrated quasiminimal pregeometry structure. Then, after adding parameters, for every $n\geq 3$ there exists an $\cl$-$n$-gon in $M$.
\end{prop}

\begin{proof} 
We do the proof for $n=4$. Working over a finite independent tuple, me may assume that
$M$ has a $\cl$-triangle $abc$. Let $a'\models \tp(a/b)$
be independent from $ac$ over $b$. Then $aca'$ is an independent triple. Let $c'$ be such that $\tp(a'c'/b) = tp(ac/b)$. Then as $b\in \cl(ac)$ and $c'\in \cl(a'b)$, the tuple $aca'c'$ is an $\cl$-$4$-gon.
\end{proof}

\begin{lem}\label{H-strucsplittingnontrivial}
Let $(M,H)$ be an $H$-structure associated to a quasiminimal class $\mathcal{C}$. Assume $cl$ is not disintegrated, then $U_{Lsp}((M,H))\geq \omega$.
\end{lem}

\begin{proof}
We will show that for every $n\geq 2$ there is an $\mathcal{L}_H$-type $p_n=p_n(x)$ in a single variable $x$ such that $U_{Lsp}(p_n)\geq n$.
By Proposition \ref{existence ngon} we can find $\{a_1,\dots,a_n,b\}\subset M$  forming a circuit. Choose $H$-structures 
$(N_0,H)\preceq(N_1,H)\preceq (N_2,H)\preceq \dots \preceq (N_n,H) \preceq (M,H)$ so that $(N_n,H)$ is $\cl$-independent from $\{a_1,\dots,a_n,b\}$ and $\dim(H(N_{i+1})/N_iH(N_i))=\infty$, $\dim(N_{i+1}/N_iH(N_{i+1}))=\infty$. Since $a_1,\dots,a_n$ are independent over $N_n$, we may assume that $a_1,\dots,a_n\in H(M)$.
Since $b\in \cl(a_1,\dots,a_n)\setminus \cl(a_1,\dots,a_{n-1})$ and $b\neq a_i$ for $1\leq i \leq n$, we have that $b\not \in H(M)$. Note that
$HB(b/N)=HB(b)=\{a_1,\dots,a_n\}$.

Consider the sequence $$\tp_H(b/N_0)\subset \tp_H(b/\cl(N_1a_1))\subset \dots \subset  \tp_H(b/\cl(N_na_1,\dots,a_n)).$$ As explained above, we see each set $\cl(N_ia_1,\dots,a_i)$ as an $H$-structure with the natural interpretation of $H$. We will show that each step is a splitting extension.
We write a proof for the first two steps, the other cases are similar.
Let $a_1'\in N_1$ satisfying $a_1'\equiv^{\mathcal{L}_H}_{N_0} a_1$ with $a_1'\not \in \{a_1,\dots,a_n\}$ (this boils down to $a_1'\in H(N_1)\setminus (H(N)\cup\{a_1\}$). Then $\tp_H(a_1/N)=\tp_H(a_1'/N)$ and since $a_1'\not \in HB(b/N)$ we have $\tp_H(b,a_1/N)\neq \tp_H(b,a_1'/N)$
and thus $\tp_H(b/N_0)\subset \tp_H(b/\cl(N_1a_1))$ is a splitting extension. Similarly, choose $a_2'\in H(N_2)\setminus H(N_1)$ with $a_2'\not \in \{a_1,\dots,a_n\}$. Then 
$\tp_H(a_2/\cl(N_1a_1))=\tp_H(a_2'/\cl(N_1a_1))$ but since $a_2'\not \in HB(b/\cl(a_1N_1))$ we have $\tp_H(b,a_2/\cl(Na_1))\neq \tp(b,a_2'/\cl(Na_1))$
and thus $\tp_H(b/\cl(a_1N_1))\subset \tp_H(b/\cl(a_1a_2N_2))$ is a splitting extension. 
\end{proof}

\begin{lem}\label{H-strucsplittingtrivial}
Let $(M,H)$ be an $H$-structure associated to a quasiminimal class $\mathcal{C}$. Assume $\cl$ is desintegrated. Then $U_{Lsp}((M,H))\leq 1$.
\end{lem}

\begin{proof}
We will show that for every single element $c\in M$ and models $(N_1,H)\preceq (N_2,H)\preceq (M,H)$ with $\dim(N_2/N_1H(N_2))=\infty$ and $\dim(H(N_2)/N_1)=\infty$ if $\tp_H(c/N_2)$ splits over $N_1$, then $\tp(c/N_2)$ splits over $N_1$ and thus $c\in N_2$.
We may assume that $c\not \in N_1$. Let $\vec b_1 \equiv^{\mathcal{L}_H}_{N_1} \vec b_2$ with $\vec b_1,\vec b_2\in N_2$ and assume that 
$\tp_H(c,\vec b_1/N_2)\neq \tp_H(c,\vec b_2/N_2)$. 

Case 1. Assume first that $c\ind^{\cl} H(N_2)$ so by triviality $c\ind^{\cl} N_1H(N_2)$. We may enlarge $\vec b_1,\vec b_2$ so that they become $H$-independent tuples over $N_1$. By triviality both tuples $c\vec b_1$, $c\vec b_2$ are $H$-independent over $N_1$ and thus if $\tp_H(c,\vec b_1/N_1)\neq \tp_H(c,\vec b_2/N_1)$ we also have $\tp(c,\vec b_1/N_1)\neq \tp(c,\vec b_2/N_1)$. This shows that the type $\tp(c/N_2)$ splits over $N_1$ and thus $c\in N_2$.

Case 2. Assume now that $c\in \cl(H(M))$. As before, we may enlarge $\vec b_1,\vec b_2$ so that they become $H$-independent tuples over $N_1$. Find $h\in H(M)$
so that $c\in \cl(h)$. By triviality both tuples $ch\vec b_1$, $ch\vec b_2$ are $H$-independent over $N_1$ and thus if $\tp_H(c,\vec b_1/N_1)\neq \tp_H(c,\vec b_2/N_1)$ we also have $\tp_H(c,h,\vec b_1/N_1)\neq \tp_H(c,h,\vec b_2/N_1)$ and from this we get $\tp(c,h,\vec b_1/N_1)\neq \tp(c,h,\vec b_2/N_1)$. This shows that the type $\tp(c,h/N_2)$ splits over $N_1$. Since $c\in cl(h)$ we have $\dim(c,h/N_1)=1$ and since $\tp(c,h/N_2)$ splits over $N_1$, we have $\dim(c,h/N_2)=0$, so $c\in N_2$.
\end{proof}

\begin{cor}\label{Lascarequals}
Let $(M,H)$ be an $H$-structure associated to a quasiminimal class $\mathcal{C}$. 
\begin{enumerate}
    \item Assume $\cl$ is disintegrated. Then Lascar splitting in the sense of $\mathcal{L}$ agrees with Lascar splitting in the sense of $\mathcal{L}_H$.
    \item Assume $\cl$ is not disintegrated. Then, in general, Lascar splitting in the sense of $\mathcal{L}_H$ does not imply Lascar splitting in the sense of $\mathcal{L}$.
\end{enumerate}

\end{cor}

\begin{proof}

(1) It follows from the proof of Lemma \ref{H-strucsplittingtrivial}.

(2) It follows from the proof of Lemma \ref{H-strucsplittingnontrivial}, Lascar splitting in $\mathcal{L}_H$ can be witnessed by decreasing the size of the $H$-basis, which may not imply 
Lascar splitting in $\mathcal{L}$. 
\end{proof}
 
The reader may want to revisit Example \ref{Torsionfreegroups} and Example \ref{equiv-classesrevisited}.
We can now state the main result of this section for $H$-structures. It is a direct consequence of the previous two lemmas.

\begin{theo}\label{ranksHstruct}
 Let $(M,H)$ be an $H$-structure associated to a quasiminimal class $\mathcal{C}$. Then $\cl$ is trivial iff $U_{Lsp}((M,H))= 1$ and $\cl$ is not trivial iff $U_{Lsp}((M,H))\geq \omega$. 
\end{theo}

Now we switch to beautiful pairs of quasiminimal pregeometry structures. The rank of the pair will discriminate between triviality, non-trivial linearity and non-linearity of the pregeometry associated to $\cl$. We start with triviality. We need a preliminary result:

\begin{lem}\label{TrivialPindep}
Let $(M,P)$ be a beautiful pair associated to a quasiminimal class $\mathcal{C}$. Assume $\cl$ is trivial. Then any tuple $\vec a\in M^n$ is $P$-independent.
\end{lem}

\begin{proof}
Write $\vec a=(a_1,\dots,a_n)$. Then for each $i\leq n$ either $a_i\in P(M)$ or $a_i\ind P(M)$, so each singleton is $P$-independent. By triviality the full tuple $\vec a$ is also $P$-independent.
\end{proof}

\begin{lem}\label{Urktrivial}
Let $(M,P)$ be a beautiful pair associated to a quasiminimal class $\mathcal{C}$. Assume $\cl$ is trivial. Then $U_{Lsp}((M,P))= 1$.
\end{lem}

\begin{proof}
Consider a chain of models $(M_1,P)\preceq (M_2,P_2)\preceq (M,P)$ where the dimension of the interpretation $P$ increases by at least $\omega$ as well as the dimension after localizing at $P$ in each step. Let $c\in M$ and assume that 
$\tp_P(c/M_2)$ splits over $M_1$. Then there are $\vec a, \vec a' \in M_2$ with $\tp_P(\vec a/M_1)= \tp_P(\vec a'/M_1)$ but $\tp_P(c,\vec a/M_1)\neq \tp_P(c,\vec a'/M_1)$. By Lemma \ref{TrivialPindep} the tuples $c,\vec a$ and $c,\vec a'$ are both $P$-independent and for each $i\leq n$ we have $P(a_i)$ iff $P(a_i')$ so we must have $\tp(c,\vec a/M_1)\neq \tp(c,\vec a'/M_1)$. Then $\tp(c/M_2)$ splits over $M_1$ and thus $c\in M_2$. In this setting Lascar splitting in the expansion agrees with Lascar splitting in the original class and the $U_{Lsp}$-rank of a singleton is at most $1$.
\end{proof}

To deal with the other cases we need to understand Lascar splitting for elements that belong to the predicate.

\begin{lem}\label{Psplitting}
Let $(M_1,P)\preceq (M,P)$ be beautiful pairs associated to a quasiminimal class $\mathcal{C}$. Let $c\in P(M)$, then $U_{Lsp}(\tp_P(c/M_1))\leq 1$.
\end{lem}

\begin{proof} Assume we have a chain of models $(M_1,P)\preceq (M_2,P_2)\preceq (M,P)$ and that $\tp_P(c/M_2)$ splits over $M_1$. Then there are $\vec a, \vec a' \in M_2$ with $\tp_P(\vec a/M_1)= \tp_P(\vec a'/M_1)$ but $\tp_P(c,\vec a/M_1)\neq \tp_P(c,\vec a'/M_1)$. We may assume, after enlarging $\vec a,\vec a'$ if necessary, that both tuples are $P$-independent over $M_1$. Note that for each $i\leq n$ we have $P(a_i)$ iff $P(a_i')$. Since $c\in P$ we must have that
$c\vec a$ is $P$-independent over $M_1$ 
and $c\vec a'$ is $P$-independent over $M_1$, so we must have $\tp(c,\vec a/M_1)\neq \tp(c,\vec a'/M_1)$. Then $\tp(c/M_2)$ splits over $M_1$ and thus $c\in M_2$. In particular, the $U_{Lsp}$-rank of a type of a singleton in $P$ is at most $1$.
\end{proof}

\begin{lem}\label{Urknottrivial}
Let $(M,P)$ be a beautiful pair associated to a quasiminimal class $\mathcal{C}$. Assume $cl$ is not trivial. Then $U_{Lsp}((M,P))\geq 2$.
\end{lem}

\begin{proof}
Consider models $(M_1,P)\preceq (M_2,P)\preceq (M,P)$ where the dimension of the interpretation of $P$ increases by at least $\omega$ in each step as well as the dimension after localizing at $P$. Let $c\in M_2\setminus (M_1\cup P(M_2))$ and let $a\in P(M)\setminus P(M_2)$. Choose $b\in M$ so that $a,b,c$ forms a circuit (a $\cl$-triangle) over $M_1$.

\textbf{Claim} The element $b$ satisfies $b\not \in P(M)$. 

Assume otherwise, then $a,b\in P(M)$. Since $\cl(P(M))=P(M)$ we get $c\in P(M)$, a contradiction.

Consider the chain $\tp_P(b/M_1)\subseteq \tp_P(b/M_2)\subseteq \tp_P(b/M)$. We will show this is a Lascar splitting chain and the lemma will follow.

\textbf{Claim} $\tp_P(b/M)$ is a splitting extension of $\tp_P(b/M_2)$. Choose an element $a'\in P(M)\setminus P(M_2)$ which is generic with respect to $b\cup M_2$, that is, $a'\not \in \cl(bM_2)$. Then $\tp_P(a/M_2)=\tp_P(a'/M_2)$ as they satisfy the unique generic type, in the predicate, over $P(M_2)$.
But $\tp_P(b,a/M_2)\neq \tp_P(b,a'/M_2)$, the first type says there is a circuit with $c\in M_2$, while this is not the case with the second type as $a'$ is generic
over $bM_2$. Note that this step $\mathcal{L}_P$-Lascar splitting corresponds to Lascar splitting in $\mathcal{C}$ as $b\in M\setminus M_2$.

\textbf{Claim} $\tp_P(b/M_2)$ is a splitting extension of $\tp_P(b/M_1)$. Choose $c'\in M_2\setminus \cl(M_1\cup P(M_2))$ with $c'\not \in \cl(cP(M_2))$. Note that $\tp_P(c'/M_1)=\tp_P(c/M_1)$, as both satisfy the unique generic type $\mathcal{L}_P$-independent from $M_1P(M_2)$.
We will prove that $\tp_P(b,c/M_1)\neq \tp_P(b,c'/M_1)$. Note that the first type says that there is an element in $P(M)$ (namely $a$) such that $(a,b,c)$ forms a circuit. Assume, in order to obtain a contradiction, that the same happens with $b,c'$, so there is $a'\in P(M)$ such that $(a',b,c')$ forms a circuit. 
Then $(M,P(M))\models \exists a\in P \exists a'\in P c\in \cl(c',a,a')$. As $c,c'\in M_2$ and $(M_2,P)\preceq (M,P)$ there are $a_1,a_1'\in P(M_2)$ with $c'\in  \cl(c,a_1,a_1')$, a contradiction with the choice of $c'$. Thus $\tp_P(b,c/M_1)\neq \tp_P(b,c'/M_1)$ and thus $\tp(b/M_2)$ Lascar splits over $M_1$.
\end{proof}

\begin{lem}\label{Urkmodular}
Let $(M,P)$ be a beautiful pair associated to a quasiminimal modular class $\mathcal{C}$. Then $U_{Lsp}((M,P))\leq 2$.
\end{lem}

\begin{proof}
Consider models $(M_1,P)\preceq (M_2,P)\preceq (M_3,P)\preceq (M,P)$ where the dimension of the interpretation $P$ increases by at least $\omega$ in each step as well as the dimension after localizing at $P$. Let  $b\in M$ and assume that $\tp_P(b/M_2)$ splits over $M_1$. Since $\tp_P(b/M_2)$ splits over $M_1$, by Theorem \ref{indepegreessplittingP} we must have $b\nind^{\cl,P}_{M_1}M_2$ and we can consider two cases.

Case 1. $\dim(b/M_2)=0$ and $\dim(b/M_1)=1$. Then $b\in M_2$ and the type $\tp(b/M_3)$ can not split over $M_2$.

Case 2. $\dim_{\scl}(b/M_2)=0$ and $\dim_{\scl}(b/M_1)=1$.

If $\tp_P(b/M_3)$ splits over $M_2$, by Theorem \ref{indepegreessplittingP} we must have $b\nind^{\cl,P}_{M_1}M_2$. Since we have $\dim_{\scl}(b/M_2)=0=\dim_{\scl}(b/M_3)=0$, we must have that 
$\dim(b/M_2)=1$ and $\dim(b/M_3)=0$ and so $b\in M_3$ and $\tp(b/M)$ can not split over $M_3$.
\end{proof}

We now extend the previous bounds to the weakly 1-based setting (see Definition \ref{defn:w1b}) as was done in \cite{v2010} for $SU$-rank-1 setting.
We need Theorem \ref{w1bimplieslocmod} and the results of section \ref{sec:linearity}.

\begin{cor}\label{rankboundedby2}
Assume the pregeometries $(M,\cl)$ from the class $\mathcal{C}$ are weakly $1$-based. Then $U_{Lsp}((P,M))\leq 2$. 
\end{cor}

\begin{proof}
    It follows from Lemma \ref{Urkmodular} after localizing in a single generic element and applying Theorem 
    \ref{w1bimplieslocmod}.
\end{proof}

The previous corollary, together with Lemma \ref{Urknottrivial} and Lemma \ref{Urktrivial}
give the other main result relating ranks and geometry for beautiful pairs:

\begin{theo}\label{ranksbeautifulpairs} Let $(M,P)$ be a beautiful pair associated to a quasiminimal class $\mathcal{C}$. Assume $(M,\cl)$ is disintegrated, then $U_{Lsp}((M,P))= 1$. Assume $(M,\cl)$ is locally modular not disintegrated, then $U_{Lsp}((M,P))= 2$.
\end{theo}

The previous result shows that the class of beautiful pairs is a FUR‐class in the sense of \cite{Ka}. Next we deal with a generalizations of the following result from \cite{v2003}:

\begin{fact}\label{eqlinSUrk1}(Theorem 5.13 \cite{v2003}) Let $T$ be a supersimple SU-rank 1 theory (with quantifier elimination). Let $T_P$ be the corresponding common theory of lovely pairs of models of $T$. Then the following are equivalent:
\begin{enumerate}
    \item[(i)] $\acl = \acl_{L_P}$ in models of $T_P$.
    \item[(ii)] $T_P$ has SU-rank $\leq 2$ ($=2$ iff $T$ is non-trivial)
    \item[(iii)] For some (any) lovely pair $(M, P)$, the localization of $(M, \acl_L)$ at $P(M)$ is
modular.
\item[(iv)] The theory $T$ is linear.
\item[(v)] The theory $T_P$ is model complete.
\end{enumerate}
\end{fact}

To find an analogue to (i) in Fact \ref{eqlinSUrk1} (condition (iii) in the statement of Thm \ref{linearity}), we follow the approach of Kangas \cite{Ka} and sustitute algebraic closure for bounded closure and work in a sufficiently large structure.

\begin{defn}
Let $\mathcal{C}$ be quasiminimal and $M\in \mathcal{C}$ sufficiently large (say $\dim(M)\geq 2^{\aleph_0}$), as usual we refer to such as structure as a \emph{monster model}. Let $\vec a\in M$ and $B\subset M$ small. We say that $\vec a$
is in the bounded closure of $B$, denoted as $\vec a\in bcl(B)$, if $tp(\vec a/B)$ has at most countably many realizations (see QM3). We write $\bcl(B)=\{a\in M: a\in \bcl(B)\}$. 

Similarly, if $(M,P)$ is sufficiently large ($\dim(P(M))\geq 2^{\aleph_0}$, $\dim(M/P(M))\geq 2^{\aleph_0}$), $\vec a\in \bcl_P(B)$, if $tp_P(\vec a/B)$ has at most countably many realizations.
\end{defn}

We need the following observation.

\begin{fact} (Lemma 77 \cite{Ka}). Let $M\in \mathcal{C}$ be a monster model for a quasiminimal class and let $B\subset M$. Then, $\cl(B)=
\bcl(B)$.
\end{fact}



The version of Fact \ref{eqlinSUrk1} that we prove is the following Theorem. Note that item (v) below gives an appropriate version of the property of model completeness in our setting.

\begin{theo}\label{linearity}
The following are equivalent for  a quasiminimal class $\mathcal{C}$:
\begin{enumerate}
    \item[(i)] $\mathcal{C}$ is weakly one-based.
    \item[(ii)] $\mathcal{C}$ is locally modular.
    \item[(iii)] In any sufficiently large beautiful pair $(M,P)$ associated to $\mathcal{C}$ we have $\bcl_P=\cl$.
    \item[(iv)] In any sufficiently large beautiful pair $(M,P)$ associated to $\mathcal{C}$ the localization $\cl(-\cup P(M))$ is modular.
    \item[(v)] (assuming that whenever $M,N\in \mathcal{C}$ and $M\subseteq N$ we have $M\preceq N$) Whenever $M,N\in \mathcal{C}$ and $(M,P)\subseteq (N,P)$ are beautiful pairs we have $(M,P)\preceq (N,P)$.
\end{enumerate}
\end{theo}


We will start by proving  (ii) iff (v).

\begin{lem}\label{modcompimplocmod}
Suppose that whenever $M,N\in \mathcal{C}$ and $M\subseteq N$ we have $M\preceq N$ (i.e. the inclusion is a closed embedding). 
 Then $\mathcal{C}$ is locally modular iff whenever $M,N\in \mathcal{C}$ and $(M,P)\subseteq (N,P)$ are beautiful pairs we have $(M,P)\preceq (N,P)$

\end{lem}    

\begin{proof}
For the first direction,  assume first that $\mathcal{C}
$ is modular and let 
$(N,P)\subseteq (M,P)$ be beautiful pairs. By Proposition \ref{H_independentsubstructurescharomega1} we only need to show that $N\ind_{P(N)}P(M)$. Note that, by hypothesis,  $N$ is closed in $M$. Since $(M,\cl)$ is modular, we have $N\ind_{N\cap P(M)}P(M)$ and since $(N,P)\subseteq (M,P)$ this means $N\ind_{P(N)}P(M)$ as desired.

  Assume now that $(M,\cl)$ is locally modular. Let $e\in P(N)$ be generic (it exists by the density property). Note that $e\in N\cap P(M)\cap P(N)$. Since $e$ is generic, by local modularity we have $N \ind_{\cl(N\cup \{e\})\cap \cl(P(M)\cup\{e\})}P(M)$ and thus
  $N\ind_{P(N)}P(M)$.

Now we prove the converse. We follow the proof of $(v)\implies (iv)$ from \cite{v2003}.
Assume $\mathcal{C}$ is not locally modular. Then for some beautiful pair $(M; P)$, the localized pregeometry $(M; \cl_P)$ is not modular. Then we can find
finite tuples $\vec a,\vec v,\vec w \in  M \setminus P(M)$, such that
$\cl(\vec a\vec w P(M)) \cap \cl(\vec v\vec w P(M)) = \cl(\vec wP(M))$ but
$\dim(\vec a/\vec v \vec w P(M)) <\dim(\vec a/\vec w P(M))$.
By exchanging $P(M)$ for $\cl(P(M)\vec w)$ we may assume $\vec w \in P(M)$ and thus construct finite tuples $\vec a, \vec v$ with $\dim(\vec a/\vec vP(M)) <\dim(\vec a/ P(M))$ and $\cl(\vec a P(M)) \cap \cl(\vec vP(M)) = \cl(P(M))$.
We may assume that we are working inside a suffiently large $N\in \mathcal{C}$. Let $M'$ be a copy of $M$ over $\cl(\vec aP(M))$ such that $M' \ind_{\cl(\vec a P(M ))}M$. Then
$\cl(\vec v P(M )) \cap M'= \cl(\vec v P(M )) \cap \cl(\vec a P(M )) = P(M)$; and thus $(M'; P(M ))\subseteq (N; \cl(\vec v P(M )))$. By assumption $M'\preceq N$. Also note that both structures are beautiful pairs, so again by assumption we have $(M'; P(M ))\preceq (N; \cl(\vec v P(M )))$ and thus also  $(M'; P(M),\vec a)\preceq (N; \cl(\vec v P(M )),\vec a)$. On the other hand, we also have $\dim(\vec a/P(M ))>\dim(\vec a/\cl(\vec v P(M )))$, a contradiction.

\end{proof}







Next, we will show that (ii) implies (iii).

\begin{lem}\label{loc-mod-implies-closure-preservation} Assume that $\mathcal{C}$ is locally modular and let $(M,P)$ be a sufficiently large beautiful pair. Then $\bcl_P=\cl$.
\end{lem}

\begin{proof}
Let $B\subseteq M$ be such that $B=\cl(B)$ and there exists $e\in \cl(B)\cap P(M)\backslash\cl(\emptyset)$. Let $a\in M\setminus B$. Assume first that $a\not \in \cl(BP(M))$, then by the codensity property $a\not \in \bcl_P(B)$. Assume now that $a\in \cl(BP(M))$ and let $\vec b\in B$, $\vec c\in P(M)$ be finite such that $a\in \cl(\vec b \vec c e)$. Since $B=\cl(B)$ and by modularity of $(M,\cl(-\cup \{e\})$, we have $B\ind_{B\cap P(M)} P(M)$, so enlarging $\vec b$ if necessary, we may assume that $\vec b$ is $P$-independent. Then by projectivity there is a single $c\in P(M)$ such that $a\in \cl(\vec b c e)$. Since $a\not\in \cl(\vec b e)$, we also have $c\not\in\cl(\vec b e)$. 

By the exchange property, we also have $c\in \cl(\vec b a e)$. By the density property $\tp(c/\vec b e)$ has unboundedly many independent realizations in $P(M)$ and since $\vec b e$ is $P$-independent, $\tp_P(c/\vec b e)$ has unboundedly many $\cl$-independent realizations. Since $c\in \cl(\vec b a e)$ then $\tp_P(a/\vec b e)$ also has unboundedly many realizations.

Thus,  whenever $B=\cl(B)$ has a non-trivial intersection with $P(M)$, we have $\bcl_P(B)=B$. If this is not the case, i.e. $B\cap P(M)=\cl(\emptyset)$, take $e_1,e_2\in P(M)\backslash \cl(\emptyset)$ independent over $B$. Then $\cl(Be_1)\cap\cl(Be_2)=B$ and $\bcl_P(Be_1)=\cl(Be_1)$, $\bcl_P(Be_2)=\cl(Be_2)$. Then $B=\cl(Be_1)\cap\cl(Be_2)=\bcl_P(Be_1)\cap\bcl_P(Be_2)$ and therefore is $\bcl_P$-closed (as an intersection of two $\bcl_P$-closed sets).

\end{proof}

The Lemma below shows that (iii) implies (iv).

\begin{lem}\label{closure-pres-implies-scl-mod}
    Suppose in every sufficiently large beautiful pair $(M,P)$ we have $\bcl_P=\cl$. Then in any such pair the localization $\cl(-\cup P(M))$ is modular.
\end{lem}

\begin{proof}
let $(M,P)$ be a sufficiently large beautiful pair.
    We will follow the proof of Theorem 5.13 ($i\to iii$) in \cite{v2003}. Let $a,b,\vec c\in M\backslash P(M)$, $\vec c$ $\cl$-independent over $P(M)$, $a\in\cl(b\vec c P(M))\backslash\cl(bP(M))$. We need to find $u\in\cl(\vec cP(M))$ such that $a\in\cl(buP(M))$.

    Let $\vec e\in P(M)$ be such that $ab\vec c\vec e$ is $P$-independent. Consider the pair $(M,P')$ where $P'(M)=\cl(\vec c P(M))$. Then $(M,P')$ is again a beautiful pair and by the assumption, it also satisfies $\bcl_P=\cl$.

Let $\vec d=d_0\ldots d_n$ be a tuple of minimal length in $P'(M)$ such that $a\in\cl(b\vec d)$ (it exists since $a\in \cl(b P'(M))$). If $n=0$, we are done as we can take $u=d_0$. If $n\ge 1$, let $\vec d'=d_1\ldots d_n$. Then by minimality $\vec d'$ is $\cl$-independent over $ab$. Since $\cl=\bcl_P$, we also have $d_i\not\in\bcl_P^{(M,P')}(d_1\ldots d_{i-1}ab)$ for $1<i\le n$. 

Then there exists $\vec d''\models \tp^{(M,P')}_P(\vec d'/ab)$, with $\vec d''=d_1''\dots d_n''$
such that $d_i''\not\in\bcl_P^{(M,P')}(d_1''\ldots d_{i-1}''ab\vec c\vec e)$ for $1<i\le n$. Thus, $\vec d''$ is $\cl$-independent over $ab\vec c\vec e$. Then, by density of $P(M)$, there exists $\vec d'''\in P(M)$ such that $\tp(\vec d'''/ab\vec c\vec e)=\tp(\vec d''/ab\vec c\vec e)$. Note that $ab\vec c\vec e$ is $P'$-independent, and both $\vec d''$ and $\vec d'''$ are in $P'(M)$. Thus $$\tp^{(M,P')}_P(\vec d'''ab\vec c\vec e)=\tp^{(M,P')}_P(\vec d''ab\vec c\vec e)=tp^{(M,P')}_P(\vec d'ab\vec c\vec e).$$

Let $d_0'\in M$ be such that $$tp^{(M,P')}_P(d_0'\vec d'''ab\vec c\vec e)=tp^{(M,P')}_P(d_0\vec d'ab\vec c\vec e).$$

Then $d_0'\in P'(M)=\cl(\vec c P(M))$ and $a\in\cl(bd_0'P(M))$, and we can take $u=d_0'$.
\end{proof}

Finally, we prove the direction $(iv \Rightarrow i)$.

\begin{lem}\label{scl-mod-implies-weak-1-based}
Suppose in every sufficiently large beautiful pair $(M,P)$ the localization $\cl(-\cup P(M))$ is modular. Then $\mathcal{C}$ is weakly 1-based.
\end{lem}
\begin{proof}
Let $M\in\mathcal{C}$, $\vec b\in M$ and $A\subset M$. We may assume that $(M,P)$ is a beautiful pair and $\vec b A\ind_{\emptyset}P(M)$. By the modularity assumption, $\vec b \ind_{\cl(\vec bP(M))\cap\cl(AP(M))}A$. Working in a large extension of $M$, we can find $\vec b'$ such that $\Lstp(\vec b'/AP(M))=\Lstp(\vec b/AP(M))$ and $\vec b'\ind_{AP(M)}\vec b$. Then $$\cl(\vec bP(M))\cap\cl(AP(M))=\cl(\vec b'P(M))\cap\cl(AP(M))$$ and, thus, $\vec b \ind_{\cl(\vec b'P(M))\cap\cl(AP(M))}AP(M)$.  We also have $\vec b\ind_{AP(M)}\vec b'P(M)$. By transitivity, $\vec b\ind_{\cl(\vec b'P(M))\cap\cl(AP(M))}\vec b' AP(M)$. Hence, $\vec b\ind_{\vec b' P(M)}A$. On the other hand, $P(M)\ind_{\emptyset} \vec b'A$, and, therefore, $\vec b' P(M)\ind_{\vec b'}\vec b' A$. Thus, $\vec b\ind_{\vec b'}A$, as needed.
\end{proof}

We are ready to prove the main result:\\
\\
\textit{Proof of Theorem \ref{linearity}.}\\
$(i\Leftrightarrow ii)$ This is the main equivalence of section 3: Theorem \ref{w1bimplieslocmod} and Theorem \ref{locmodimpliesw1b};
$(ii\Rightarrow iii)$ Lemma \ref{loc-mod-implies-closure-preservation};
$(ii\Rightarrow iv)$ is clear;
$(iii \Rightarrow iv)$ Lemma \ref{closure-pres-implies-scl-mod};
$(ii\Leftrightarrow v)$ Lemma  \ref{modcompimplocmod};
$(iv\Rightarrow i)$ Lemma \ref{scl-mod-implies-weak-1-based}.
\\
\\
There are other characterizations of local modularity that are worth exploring. For example, finding analogue of the work of Buechler \cite{Bu} and of the second named author \cite{v2010}.

\begin{ques}
  Let $(M,P)$ be a beautiful pair associated to a quasiminimal class $\mathcal{C}$. Assume $\mathcal{C}$ is not locally modular. Do we have then $U_{Lsp}((M,P))= \omega$?  
\end{ques}

Finally, we leave some questions about expansions related to the ones studied here that might be of interest.

\begin{ques}
  Assume that $(\mathbb{K},\ecl)$ is Zilber's exponential field. Expand the field structure with a multiplicative subgroup with the Mann property as was done in \cite{vdDG}. Is the collections of such pairs  $\mathcal{L}_{\omega_1\omega}(Q)$-axiomatizable? Is the associated class stable? Kirby proved (Theorem 1.1 \cite{KiEx}) that in any exponential field the closure operator
$\ecl$ is a pregeometry. We leave as an open question to explore the meaning of $H$-structures, beautiful/lovely pairs and Mann expansions in these other exponential settings.
\end{ques}

\end{document}